\documentclass[reqno]{amsart}

\usepackage[margin=1.4in]{geometry}
\usepackage{amssymb, amsmath, amsthm, amsfonts, amscd}
\usepackage{mathrsfs,mathtools}
\usepackage[shortlabels]{enumitem}
\usepackage{tikz, tikz-cd}
\usepackage[all]{xy}
\usepackage[new]{old-arrows}
\usepackage{pst-node}
\usepackage{varwidth}
\usepackage{graphicx}
\usepackage{subfig}
\usepackage{tabularx, booktabs}
\usepackage{multirow, multicol}
\usepackage{xcolor}
\usepackage{mathrsfs}
\usepackage{yfonts, euscript}
\usepackage{csquotes, dirtytalk}
\usepackage[pagewise]{lineno} 
\usepackage[colorlinks=true, allcolors=blue]{hyperref}

\usetikzlibrary{arrows, arrows.meta}

\usepackage{listings}
\usepackage{xcolor}

\newtheorem{thm}{{\bf Theorem}}[section]
\newtheorem{lemma}[thm]{{\bf Lemma}}
\newtheorem{prop}[thm]{{\bf Proposition}}

\newtheorem*{rmk}{{\bf Remark}}

\newtheorem*{ex}{{\bf Example}}

\newtheorem*{thmA}{\bf Theorem A}
\newtheorem*{thmB}{\bf Theorem B}

\newcommand{\RNum}[1]{\uppercase\expandafter{\romannumeral #1\relax}}

\newcommand{\GL}{\,\mathrm{GL}\,}
\newcommand{\SL}{\,\mathrm{SL}\,}
\newcommand{\SU}{\,\mathrm{SU}\,}
\newcommand{\PGL}{\,\mathrm{PGL}\,}

\newcommand{\diag}{{\rm diag}}

\begin{document}


\title[Automorphisms of Twisted Chevalley Groups of types ${}^2 A_3$ and ${}^2 A_4$]{Automorphisms of Twisted Chevalley Groups \\ of types ${}^2 A_3$ and ${}^2 A_4$ over Local Rings}


\author{Elena I. Bunina}
\address{Department of Mathematics,
    	Bar–Ilan University, Ramat Gan, Israel}
\email{\href{mailto:helenbunina@gmail.com}{helenbunina@gmail.com}}
\thanks{}

\author{Deep H. Makadiya}
\address{Department of Mathematics,
    	Bar–Ilan University, Ramat Gan, Israel}
\email{\href{mailto:deepmakadia25.dm@gmail.com}{deepmakadia25.dm@gmail.com}}
\thanks{}




\begin{abstract}
    We give a complete classification of the automorphisms of twisted Chevalley groups of types ${}^2A_3$ and ${}^2A_4$, and of their elementary subgroups, over commutative local rings $R$ containing $1/2$ and equipped with a non-trivial involution $\theta$.
    The classification covers all isogeny types and, together with our higher-rank result~\cite{EB&DM1}, completes the description for type ${}^2A_{\ell}$ with $\ell\geq3$.

    Unlike the higher-rank groups, the elementary groups in these two types can have genuinely nonstandard automorphisms.
    Let $J$ be the Jacobson radical of $R$ and write $R^-_\theta=\{r\in R\mid\theta(r)=-r\}$.
    The exceptional family is parametrized by
    \[
        \mathcal E(R,\theta)=
        \begin{cases}
            \operatorname{Ann}_R(J)\cap R^-_\theta,
                & \text{if }R/J\cong\mathbf F_3,\\[2mm]
            \{0\},
                & \text{if }R/J\not\cong\mathbf F_3.
        \end{cases}
    \]
    For each $\delta\in\mathcal E(R,\theta)$ and $n\in\{3,4\}$, we explicitly construct an automorphism $\Theta_\delta^{(n)}$ of $E'_{\pi,\sigma}(A_n,R)$, which is nonstandard precisely when $\delta\neq0$.
    These automorphisms arise from explicit $1$-cocycles of the residual elementary group, lifted to a square-zero congruence layer over $R$.
    They account for all failures of standardness: every automorphism of the elementary group is a composition of a strictly inner automorphism, a diagonal automorphism, a ring automorphism commuting with $\theta$, and a single exceptional automorphism $\Theta_\delta^{(n)}$.

    For the full groups, the answer depends on the isogeny type.
    A nontrivial exceptional automorphism extends to the full group precisely in the simply connected case of type ${}^2A_3$, where the full group is $\operatorname{SU}_4(R,Q_4)$ and coincides with its elementary subgroup.
    For every other isogeny type of ${}^2A_3$ and ${}^2A_4$, all automorphisms of the full group are standard: they are compositions of strictly inner, diagonal, ring, and central automorphisms.
\end{abstract}


\maketitle
\tableofcontents


\section{Introduction}\label{sec:intro}

The automorphism problem for classical and Chevalley groups asks how much of the coefficient ring and the defining geometry can be recovered from the abstract group structure.
Standardness results express a strong form of this principle: abstract automorphisms are induced by the natural symmetries of the coefficient ring, the form, or the root datum, together with inner and central automorphisms.
Over fields, the classical descriptions for linear, symplectic, orthogonal, and unitary groups were established through the work of Dieudonn\'e, Steinberg, and others; see~\cite{Dieudonne,RS} and the references therein.
The present paper identifies a low-rank exception over local rings which is invisible after reduction to the residue field, and gives a complete classification for types ${}^2A_3$ and ${}^2A_4$.

For unitary groups, the coefficient structure is a ring together with an involution, and the automorphism problem must recover both.
The theory of unitary groups over rings with involution has developed through several complementary approaches.
Vaserstein studied the stabilization of unitary and orthogonal groups~\cite{VasersteinUnitary}; Bak introduced form rings and unitary $K$-theory~\cite{Bak}; and Hahn and O'Meara gave a systematic treatment of classical groups, their Steinberg groups, $K$-groups, and isomorphism theory~\cite{HahnOMeara}.
The isomorphism problem for unitary groups over associative rings containing $1/2$ was studied by Golubchik and Mikhalev~\cite{GolubchikMikhalevUnitary}.
An important approach due to Zel'manov~\cite{ZelmanovLinearGroups} relates isomorphisms of linear groups and groups defined by involutions on matrix rings to homomorphisms of graded Lie algebras and Jordan systems.
His standardness results in the latter setting assume at least three hyperbolic pairs, whereas the groups considered here have relative rank two.

For untwisted Chevalley groups over commutative rings, automorphisms, normal subgroups, congruence properties, and related structural questions have been investigated extensively; see, among many other works,~\cite{EA0,EA2,EA3,EA4,EA&KS,GT,LV,AK,NV1,EB24} and the references therein.
Under the appropriate rank and invertibility assumptions, automorphisms admit the familiar description in terms of ring, inner, graph, and central components.
The corresponding automorphism theory for twisted Chevalley groups over rings is less complete.
Recent structural results directly relevant to this problem concern normal subgroups and congruence properties~\cite{SG&DM1}, normalizers~\cite{SG&DM2}, and triangular and unitriangular factorizations~\cite{SG&DM3}.

Building on these developments, we initiated a series on automorphisms and isomorphisms of twisted Chevalley groups over commutative rings.
In our previous work~\cite{EB&DM1}, we proved that all automorphisms of twisted Chevalley groups and their elementary subgroups of type ${}^2A_{\ell}$, $\ell\geq5$, over local rings containing $1/2$ are standard.
Here we resolve the remaining relative-rank-two cases ${}^2A_3$ and ${}^2A_4$, completing the classification for ${}^2A_{\ell}$ with $\ell\geq3$ across all isogeny types.

These two cases were excluded from the earlier paper because the higher-rank normalization arguments no longer apply.
Our results show that this is a structural difference, not merely a limitation of the method.
In both types, a nonstandard automorphism can survive in the elementary group precisely when the residue field is $\mathbf F_3$ and the anti-invariant part of the annihilator of the Jacobson radical is nonzero.
The additional automorphisms arise from residual $1$-cocycles and square-zero corrections over the ring.
We construct them explicitly and prove that they are the only obstruction to standardness: after removing the standard factors, every automorphism of the elementary group belongs to this exceptional family.
For the full group, additional torus relations may eliminate the exceptional factor, so the final answer depends on the isogeny type.

\medskip

We use the terminology and general preliminary results of~\cite[Section~2]{EB&DM1}, recalling below the notation needed to state our theorems.
The matrix realizations and structural properties used in the proofs are recorded in Section~\ref{sec:matrix-realizations}.

Throughout this paper, $R$ is a commutative local ring with $2\in R^\times$, Jacobson radical $J$, and residue field $k\coloneqq R/J$.
We fix a non-trivial involution $\theta\colon R\to R$ and write $\bar r\coloneqq\theta(r)$.
The involution induced on $k$ is allowed to be trivial; indeed, it is necessarily trivial when $k\cong\mathbf F_3$.
We write
\[
    R_\theta=\{r\in R\mid\bar r=r\},
    \qquad
    R^-_\theta=\{r\in R\mid\bar r=-r\}
\]
for the symmetric and anti-invariant parts of $R$.
For a $\theta$-invariant ideal $I\subseteq R$, we similarly put
\[
    I_\theta=I\cap R_\theta,
    \qquad
    I^-_\theta=I\cap R^-_\theta.
\]

Let $n\in\{3,4\}$, and let $G_\pi(A_n,R)$ be the Chevalley group associated with a representation $\pi$ whose weight lattice $\Lambda_\pi$ is invariant under the nontrivial graph automorphism $\rho$ of the Dynkin diagram of type $A_n$.
We use the same symbols $\rho$ and $\theta$ for the induced automorphisms of $G_\pi(A_n,R)$ and set $\sigma\coloneqq\rho\circ\theta$.
The twisted Chevalley group $G_{\pi,\sigma}(A_n,R)$ is the fixed-point subgroup under $\sigma$, and its elementary subgroup is denoted by $E'_{\pi,\sigma}(A_n,R)$.
For brevity, we write
\[
    E_\pi^{(n)}\coloneqq E'_{\pi,\sigma}(A_n,R),
    \qquad
    G_\pi^{(n)}\coloneqq G_{\pi,\sigma}(A_n,R).
\]

For $A_3$, the three possible isogeny types correspond to the lattices
\[
    \Lambda_r,
    \qquad
    \Lambda_{\mathrm{mid}}\coloneqq\Lambda_r+\mathbb Z\omega_2,
    \qquad
    \Lambda_{\mathrm{sc}},
\]
namely the adjoint, intermediate, and simply connected types.
For $A_4$, only the adjoint and simply connected types occur.
In the standard matrix realizations of Section~\ref{sec:matrix-realizations}, the simply connected groups are
\[
    G_{\mathrm{sc}}^{(3)}=\operatorname{SU}_4(R,Q_4),
    \qquad
    G_{\mathrm{sc}}^{(4)}=\operatorname{SU}_5(R,Q_5).
\]

We use the definitions of \emph{inner}, \emph{strictly inner}, \emph{diagonal}, \emph{ring}, and \emph{central automorphisms} from~\cite[Section~2 and~9]{EB&DM1}.
An automorphism of $G_\pi^{(n)}$ or $E_\pi^{(n)}$ is called \emph{standard} if it is a composition of inner, diagonal, ring, and central automorphisms.
There is no separate graph factor here: on the fixed-point group, $\rho$ agrees with the automorphism induced by $\theta$.

The exceptional family is governed by the parameter space
\[
    \mathcal E(R,\theta)=
    \begin{cases}
        \operatorname{Ann}_R(J)\cap R^-_\theta,
            & R/J\cong\mathbf F_3,\\[2mm]
        \{0\},
            & R/J\not\cong\mathbf F_3.
    \end{cases}
\]
An equivalent annihilator criterion is given in Subsection~\ref{subsec:expl_prm_spc}.
For example, if $R=\mathbf F_3[x]/(x^2)$ and $\theta(x)=-x$, then $\mathcal E(R,\theta)=\mathbf F_3 \, x \neq0$.

For each $\delta\in\mathcal E(R,\theta)$, Subsections~\ref{subsec:expl_elt} and~\ref{subsec:expl_aut} construct an automorphism
\[
    \Theta_\delta^{(n)}\in\operatorname{Aut}(E_\pi^{(n)}),
    \qquad n\in\{3,4\}.
\]
Its cohomological origin is most transparent in the simply connected matrix model.
When $k\cong\mathbf F_3$, we construct explicit maps
$c_n\colon E_{\mathrm{sc}}^{(n)}(k)\to M_{n+1}(k)$ satisfying the $1$-cocycle identity for conjugation,
\[
    c_n(gh)=c_n(g)+g\,c_n(h)\,g^{-1}.
\]
The corresponding automorphism over $R$ is given by
\[
    \Theta_\delta^{(n)}(A)
    =\bigl(I_{n+1}+\delta\,\widehat{c_n(A_0)}\bigr)A,
    \qquad A\in E_{\mathrm{sc}}^{(n)},
\]
where $A_0$ is the reduction of $A$ modulo $J$ and the hat denotes any matrix lift to $R$.
The conditions $\delta J=0$ and $\delta^2=0$ make this formula independent of the lift and turn the cocycle identity into the homomorphism property.
For the chosen cocycles, the maps preserve the elementary group and fix its center pointwise, and hence descend to every isogeny type.
In the simply connected model, the formula also shows that the extra symmetry acts trivially on the residual group and fixes the kernel of reduction pointwise, although it need not be trivial on the group over $R$.

The exceptional automorphisms satisfy
\[
    \begin{gathered}
        \Theta_\delta^{(n)}\circ\Theta_\varepsilon^{(n)}
        =\Theta_\varepsilon^{(n)}\circ\Theta_\delta^{(n)}
        =\Theta_{\delta+\varepsilon}^{(n)},\\[2mm]
        \Theta_0^{(n)}=\operatorname{id}_{E_\pi^{(n)}},
        \qquad
        \bigl(\Theta_\delta^{(n)}\bigr)^{-1}
        =\Theta_{-\delta}^{(n)}.
    \end{gathered}
\]
Moreover, $\Theta_\delta^{(n)}$ is nonstandard if and only if $\delta\neq0$; see Proposition~\ref{prop:exceptional-nonstandard}.

\medskip

We can now state the main results.
The first theorem classifies the automorphisms of the elementary groups for every isogeny type.

\begin{thmA}
    Let $R$ be a commutative local ring with a non-trivial involution $\theta$ such that $1/2\in R$.
    Let $n\in\{3,4\}$, and let $\pi$ be any isogeny type described above.
    Then every automorphism $\varphi\in\operatorname{Aut}(E_\pi^{(n)})$ has the form
    \[
        \varphi=i_g\circ d\circ\mu\circ\Theta_\delta^{(n)},
    \]
    where $i_g$ is a strictly inner automorphism, $d$ is a diagonal automorphism, $\mu$ is a ring automorphism commuting with $\theta$, and $\delta\in\mathcal E(R,\theta)$.

    Conversely, every composition of this form is an automorphism of $E_\pi^{(n)}$.
\end{thmA}

In particular, all automorphisms of $E_\pi^{(n)}$ are standard if and only if $\mathcal E(R,\theta)=\{0\}$.
For the full groups, nontrivial exceptional automorphisms occur only in one isogeny type.

\begin{thmB}
    Let $R$, $n$, and $\pi$ be as in Theorem~A.
    \begin{enumerate}
        \item For the simply connected full group of type ${}^2A_3$, every automorphism $\varphi\in\operatorname{Aut}(G_{\mathrm{sc}}^{(3)})$ has the form
        \[
            \varphi=i_g\circ d\circ\mu\circ\Theta_\delta^{(3)},
        \]
        where $i_g$ is a strictly inner automorphism, $d$ is a diagonal automorphism, $\mu$ is a ring automorphism commuting with $\theta$, and $\delta\in\mathcal E(R,\theta)$.

        \item For
        \[
            (n,\pi)\in
            \bigl\{
                (3,\mathrm{mid}),
                (3,\mathrm{ad}),
                (4,\mathrm{sc}),
                (4,\mathrm{ad})
            \bigr\},
        \]
        every automorphism of $G_\pi^{(n)}$ is standard.
        More precisely, every $\varphi\in\operatorname{Aut}(G_\pi^{(n)})$ has the form
        \[
            \varphi=i_g\circ d\circ\mu\circ\tau,
        \]
        where $i_g$ is a strictly inner automorphism, $d$ is a diagonal automorphism, $\mu$ is a ring automorphism commuting with $\theta$, and $\tau$ is a central automorphism.
    \end{enumerate}

    Conversely, every composition of the indicated form is an automorphism of the corresponding full group.
\end{thmB}

The distinction in Theorem~B is forced by a torus symmetry.
For each full group in part~(2), there is a torus element whose conjugation on the elementary subgroup sends $\Theta_\delta^{(n)}$ to $\Theta_{-\delta}^{(n)}$.
When $\mathcal E(R,\theta)\neq\{0\}$, this element does not belong to the elementary subgroup.
An extension of $\Theta_\delta^{(n)}$ to the full group would force $\Theta_{2\delta}^{(n)}$ to be standard, and hence $\delta=0$; see Proposition~\ref{prop:exceptional-does-not-extend}.
In contrast, $G_{\mathrm{sc}}^{(3)}=\operatorname{SU}_4(R,Q_4)$ coincides with its elementary subgroup.
Thus the exceptional family gives nonstandard automorphisms of the classical group $\operatorname{SU}_4(R,Q_4)$ itself, while the other full groups remain standard.

The paper is organized as follows.
Section~\ref{sec:matrix-realizations} records all the necessary preliminaries: the matrix realizations, relative root elements, Weyl representatives, torus elements, and matrix-algebra generation properties used throughout the proofs.
In Section~\ref{sec:exceptional_aut} we construct the exceptional automorphisms and establish their main properties.
Section~\ref{sec:aut_ele_adj_local} classifies the automorphisms of the adjoint elementary groups; this is where the main classification argument is carried out.
Finally, Section~\ref{sec:arbitrary-isogeny-types} treats the elementary and full groups of all isogeny types, proving Theorems~A and~B.


\section{Standard Matrix Realizations}\label{sec:matrix-realizations}

We retain the notation and assumptions of Section~\ref{sec:intro}.
Here we record the matrix realizations, relative root elements, Weyl representatives, and torus elements used in the proofs; see~\cite{EB&DM1} for the general definitions.

For $n\in\{3,4\}$, we use the standard simple roots of $A_n$ and the graph involution
\[
    \alpha_i=\varepsilon_i-\varepsilon_{i+1},
    \qquad
    \rho(\alpha_i)=\alpha_{n+1-i}
    \qquad (1\leq i\leq n).
\]
Over the local ring $R$, the untwisted simply connected and adjoint groups are
\[
    G_{\mathrm{sc}}(A_n,R)\cong\SL_{n+1}(R),
    \qquad
    G_{\mathrm{ad}}(A_n,R)\cong\PGL_{n+1}(R)
    \cong\GL_{n+1}(R)/R^\times.
\]
We describe below the corresponding fixed-point groups under $\sigma=\rho\circ\theta$.
The bar acts entrywise on matrices, and $E_{ij}$ denotes the standard matrix unit.
We abbreviate $\SU_{n+1}(R,Q_{n+1})$ to $\SU_{n+1}(R)$ for the forms specified below.


\subsection{Type \texorpdfstring{${}^{2}A_3$}{2A3}}

Put
\[
    \beta_1\coloneqq[\alpha_1],\qquad
    \beta_2\coloneqq[\alpha_2],\qquad
    \beta_3\coloneqq\beta_1+\beta_2,\qquad
    \beta_4\coloneqq2\beta_1+\beta_2,
\]
where $[\alpha_1]=\{\alpha_1,\alpha_3\}$ and $[\alpha_2]=\{\alpha_2\}$.
The twisted root system
\[
    \Phi_\rho=\{\pm\beta_1,\pm\beta_2,\pm\beta_3,\pm\beta_4\}
\]
is of type $B_2$.
The short roots $\pm\beta_1,\pm\beta_3$ are of type $A_1^2$, and the long roots $\pm\beta_2,\pm\beta_4$ are of type $A_1$.

\begin{center}
    \begin{tikzpicture}[scale=1.5, >=stealth, thick]
        \draw[->] (0,0) -- (1,0) node[right, xshift=2pt] {$\beta_1$};
        \draw[->] (0,0) -- (0,1) node[above, yshift=2pt] {$\beta_1 + \beta_2$};
        \draw[->] (0,0) -- (-1,0) node[left, xshift=-2pt] {$-\beta_1$};
        \draw[->] (0,0) -- (0,-1) node[below, yshift=-2pt] {$-\beta_1-\beta_2$};

        \draw[->] (0,0) -- (1,1) node[above right] {$2\beta_1 + \beta_2$};
        \draw[->] (0,0) -- (-1,1) node[above left] {$\beta_2$};
        \draw[->] (0,0) -- (-1,-1) node[below left] {$-2\beta_1 - \beta_2$};
        \draw[->] (0,0) -- (1,-1) node[below right] {$-\beta_2$};
    \end{tikzpicture}
\end{center}

Since $\Lambda_{\mathrm{sc}}/\Lambda_r\cong\mathbb Z/4\mathbb Z$, the possible isogeny types are $\mathrm{ad}$, $\mathrm{mid}$, and $\mathrm{sc}$, corresponding to the lattices described in Section~\ref{sec:intro}.
The simply connected and adjoint twisted groups have the matrix realizations
\begin{align*}
    G_{\mathrm{sc}}^{(3)} &\cong \SU_4(R) \coloneqq \{ A \in \SL_4(R) \mid A Q_4 \bar{A}^{t} = Q_4 \}, \\
    G_{\mathrm{ad}}^{(3)} &\cong \{ [A] \in \PGL_4(R) \mid A Q_4 \bar{A}^{t} = \lambda \, Q_4 \text{ for some } \lambda \in R^{\times} \},
\end{align*}
where
\[
    Q_4 = \begin{pmatrix}
        0 & 0 & 0 & 1 \\
        0 & 0 & -1 & 0 \\
        0 & 1 & 0 & 0 \\
        -1 & 0 & 0 & 0
    \end{pmatrix}, \qquad
    Q_4^{t} = -Q_4, \qquad
    Q_4^2 = -I_4.
\]

\smallskip

For $t\in R$ and $s\in R_\theta$, the positive relative root elements are
\begin{align*}
    x_1(t) &\coloneqq x_{\beta_1}(t) = I_4 + t E_{12} + \bar{t} \, E_{34}, &
    x_2(s) &\coloneqq x_{\beta_2}(s) = I_4 + s E_{23}, \\
    x_3(t) &\coloneqq x_{\beta_3}(t) = I_4 - t E_{13} + \bar{t} \, E_{24}, &
    x_4(s) &\coloneqq x_{\beta_4}(s) = I_4 - s E_{14}.
\end{align*}
The negative root elements are defined by $x_{-i}(a)\coloneqq x_i(a)^t$, with the same parameter domains.

The Weyl and torus elements are defined by
\[
    w_i(t) \coloneqq w_{\beta_i}(t) = x_{i}(t) \, x_{-i}(-t^{-1}) \, x_{i}(t)
    \qquad \text{and} \qquad
    h_i(t) \coloneqq h_{\beta_i}(t) = w_{i}(t) \, w_{i}(-1),
\]
where $t \in R^{\times}$ for $i \in \{1, 3\}$, and $t \in R_{\theta}^{\times}$ for $i \in \{2, 4\}$. Explicitly,
\begin{align*}
    w_1(t) &= \begin{pmatrix}
        0 & t & 0 & 0 \\
        -1/t & 0 & 0 & 0 \\
        0 & 0 & 0 & \bar{t} \\
        0 & 0 & -1/\bar{t} & 0
    \end{pmatrix}, & w_2(t) &= \begin{pmatrix}
        1 & 0 & 0 & 0 \\
        0 & 0 & t & 0 \\
        0 & -1/t & 0 & 0 \\
        0 & 0 & 0 & 1
    \end{pmatrix}, \\
    w_3(t) &= \begin{pmatrix}
        0 & 0 & -t & 0 \\
        0 & 0 & 0 & \bar{t} \\
        1/t & 0 & 0 & 0 \\
        0 & -1/\bar{t} & 0 & 0
    \end{pmatrix}, & w_4(t) &= \begin{pmatrix}
        0 & 0 & 0 & -t \\
        0 & 1 & 0 & 0 \\
        0 & 0 & 1 & 0 \\
        1/t & 0 & 0 & 0
    \end{pmatrix};
\end{align*}
and
\begin{align*}
    h_1(t) &= \diag(t,1/t,\bar{t},1/\bar{t}), & h_2(t) &= \diag(1,t,1/t,1), \\
    h_3(t) &= \diag(t, \bar{t}, 1/t, 1/\bar{t}), & h_4(t) &= \diag(t,1,1,1/t).
\end{align*}

These root elements generate the elementary group
\[
    E_{\mathrm{sc}}^{(3)} = \big\langle x_{\pm1}(t), \, x_{\pm3}(t), \, x_{\pm2}(s), \, x_{\pm4}(s) \mid t \in R, \ s \in R_\theta \big\rangle \subseteq G_{\mathrm{sc}}^{(3)}.
\]
For any isogeny type $\pi$, the group $E_\pi^{(3)}$ is the natural image of $E_{\mathrm{sc}}^{(3)}$ in $G_\pi^{(3)}$.

We use the abbreviations
\[
    u_{\pm i}\coloneqq x_{\pm i}(1),\qquad
    w_i\coloneqq w_i(1),\qquad
    h_i\coloneqq h_i(-1)
    \qquad (i=1,\ldots,4).
\]
The following identities will be used repeatedly:
\begin{gather*}
    w_1^2 = w_3^2 = h_1 = h_3 = - I_4, \\
    w_2^2 = - w_4^2 = h_2 = -h_4 = \diag(1,-1,-1,1), \\
    w_2 u_1 w_2^{-1} = u_3, \qquad w_1 u_2 w_1^{-1} = u_4.
\end{gather*}

For later use, write
\[
    \mathcal{R}^{(3)}_i = \begin{cases}
        R & \text{if } i \in \{1, 3\}, \\
        R_\theta & \text{if } i \in \{2, 4\},
    \end{cases}
    \qquad \text{and} \qquad
    \big(\mathcal{R}^{(3)}_i\big)^\times = \begin{cases}
        R^{\times} & \text{if } i \in \{1, 3\}, \\
        R_\theta^{\times} & \text{if } i \in \{2, 4\}.
    \end{cases}
\]

\begin{lemma} \label{lemma:A3-generates-matrix-algebra}
    The matrices $u_{\pm 1}$ and $u_{\pm 2}$ generate the full matrix algebra $M_4(R)$ over $R$.
\end{lemma}

\begin{proof}
    Set
    \begin{align*}
        A & = u_{1}-I_4 = E_{12}+E_{34}, &
        B & = u_{-1}-I_4 = E_{21}+E_{43},\\
        C & = u_{2}-I_4 = E_{23}, &
        D & = u_{-2}-I_4 = E_{32}.
    \end{align*}
    Then $CD=E_{22}$ and $DC=E_{33}$.
    Hence
    \[
        A E_{22} = E_{12},
        \quad
        E_{33} A = E_{34},
        \quad
        E_{22} B = E_{21},
        \quad
        B E_{33} = E_{43}.
    \]
    Thus all adjacent matrix units belong to the generated algebra.
    Their products give every $E_{ij}$, proving the assertion.
\end{proof}


\subsection{Type \texorpdfstring{${}^{2}A_4$}{2A4}}

In the twisted-root notation of~\cite{EB&DM1}, put
\[
    \beta_1\coloneqq[\alpha_1],\qquad
    \beta_2\coloneqq[\alpha_2],\qquad
    \beta_3\coloneqq\beta_1+\beta_2,\qquad
    \beta_4\coloneqq\beta_1+2\beta_2,
\]
where $[\alpha_1]=\{\alpha_1,\alpha_4\}$ and
$[\alpha_2]=\{\alpha_2,\alpha_3,\alpha_2+\alpha_3\}$.
The system
\[
    \Phi_\rho=\{\pm\beta_1,\pm\beta_2,\pm\beta_3,\pm\beta_4\}
\]
has type $B_2$.
Here the long roots $\pm\beta_1,\pm\beta_4$ are of type $A_1^2$, while the short roots $\pm\beta_2,\pm\beta_3$ are of type $A_2$.

\begin{center}
    \begin{tikzpicture}[scale=1.5, >=stealth, thick]
        \draw[->] (0,0) -- (1,0) node[right, xshift=2pt] {$\beta_2$};
        \draw[->] (0,0) -- (0,1) node[above, yshift=2pt] {$\beta_1 + \beta_2$};
        \draw[->] (0,0) -- (-1,0) node[left, xshift=-2pt] {$-\beta_2$};
        \draw[->] (0,0) -- (0,-1) node[below, yshift=-2pt] {$-\beta_1-\beta_2$};

        \draw[->] (0,0) -- (1,1) node[above right] {$\beta_1 + 2\beta_2$};
        \draw[->] (0,0) -- (-1,1) node[above left] {$\beta_1$};
        \draw[->] (0,0) -- (-1,-1) node[below left] {$-\beta_1 - 2\beta_2$};
        \draw[->] (0,0) -- (1,-1) node[below right] {$-\beta_1$};
    \end{tikzpicture}
\end{center}

Since $\Lambda_{\mathrm{sc}}/\Lambda_r\cong\mathbb Z/5\mathbb Z$, only the simply connected and adjoint isogeny types occur.
Their twisted groups have the matrix realizations
\begin{align*}
    G_{\mathrm{sc}}^{(4)} &\cong \SU_5(R) \coloneqq \{ A \in \SL_5(R) \mid A Q_5 \bar{A}^{t} = Q_5 \}, \\
    G_{\mathrm{ad}}^{(4)} &\cong \{ [A] \in \PGL_5(R) \mid A Q_5 \bar{A}^{t} = \lambda \, Q_5 \text{ for some } \lambda \in R^{\times} \},
\end{align*}
where
\[
    Q_5 = \begin{pmatrix}
        0 & 0 & 0 & 0 & 1 \\
        0 & 0 & 0 & -1 & 0 \\
        0 & 0 & 1 & 0 & 0 \\
        0 & -1 & 0 & 0 & 0 \\
        1 & 0 & 0 & 0 & 0
    \end{pmatrix}, \qquad
    Q_5^{t} = Q_5, \qquad
    Q_5^2 = I_5.
\]

For $t\in R$, the positive long-root elements are
\[
    x_1(t) \coloneqq x_{\beta_1}(t) = I_5 + t E_{12} + \bar{t} \, E_{45}, \quad
    x_4(t) \coloneqq x_{\beta_4}(t) = I_5 + t E_{14} + \bar{t} \, E_{25}.
\]
The short-root parameter group is
\[
    \mathcal{A}(R) = \{(t,u) \in R^2 \mid t \bar{t} = u + \bar{u} \},
\]
with the Heisenberg-type operation
\[
    (t,u) \oplus (s,v) \coloneqq (t+s, \, u+v+t\bar{s}).
\]
For $(t,u)\in\mathcal A(R)$, the positive short-root elements are
\begin{align*}
    x_2(t,u) &\coloneqq x_{\beta_2}(t,u) = I_5 + t E_{23} + \bar{t} E_{34} + u E_{24}, \\
    x_3(t,u) &\coloneqq x_{\beta_3}(t,u) = I_5 + t E_{13} - \bar{t} E_{35} - u E_{15}.
\end{align*}
The negative root elements are again the transposes of the positive ones, with the same parameters.

For the long roots $i\in\{1,4\}$ and $t\in R^\times$, set
\[
    w_i(t) \coloneqq w_{\beta_i}(t) = x_{i}(t) \, x_{-i}(-t^{-1}) \, x_{i}(t)
    \qquad \text{and} \qquad
    h_i(t) \coloneqq h_{\beta_i}(t) = w_{i}(t) \, w_{i}(-1).
\]
For the short roots $i\in\{2,3\}$, set
\[
    w_i(t,u) \coloneqq w_{\beta_i}(t,u) = x_{i}(t,u) \, x_{-i}(-\bar{t} \bar{u}^{-1}, \bar{u}^{-1}) \, x_{i}(tu^{-1}\bar{u}, u)
\]
for $(t,u)\in\mathcal A(R)^\times\coloneqq\{(t,u)\in\mathcal A(R)\mid u\in R^\times\}$, and define
\[
    h_i\big((t_1, u_1), (t_2, u_2)\big) \coloneqq w_i(t_1, u_1) w_i(t_2, u_2)
\]
for all $(t_1, u_1), (t_2, u_2) \in \mathcal{A}(R)^{\times}$.

These root elements generate the elementary group
\[
    E_{\mathrm{sc}}^{(4)} = \big\langle x_{\pm1}(t), \, x_{\pm4}(t), \, x_{\pm2}(t,u), \, x_{\pm3}(t,u) \mid t \in R, \ (t,u) \in \mathcal{A}(R) \big\rangle \subseteq G_{\mathrm{sc}}^{(4)}.
\]
Its projective image in $G_{\mathrm{ad}}^{(4)}$ is $E_{\mathrm{ad}}^{(4)}$.

We use the abbreviations
\[
    u_{\pm i}\coloneqq x_{\pm i}(1)\quad(i=1,4),
    \qquad
    u_{\pm i}\coloneqq x_{\pm i}(1,1/2)\quad(i=2,3),
\]
and
\[
    w_i\coloneqq w_i(1)\quad(i=1,4),
    \qquad
    v_i\coloneqq w_i(1,1/2)\quad(i=2,3).
\]
Also put
\[
    H\coloneqq h_1(-1)=h_4(-1)=\diag(-1,-1,1,-1,-1).
\]
The distinguished Weyl representatives are
\begin{align*}
    w_1 &= \begin{pmatrix}
        0 & 1 & 0 & 0 & 0 \\
        -1 & 0 & 0 & 0 & 0 \\
        0 & 0 & 1 & 0 & 0 \\
        0 & 0 & 0 & 0 & 1 \\
        0 & 0 & 0 & -1 & 0
    \end{pmatrix}, &
    w_4 &= \begin{pmatrix}
        0 & 0 & 0 & 1 & 0 \\
        0 & 0 & 0 & 0 & 1 \\
        0 & 0 & 1 & 0 & 0 \\
        -1 & 0 & 0 & 0 & 0 \\
        0 & -1 & 0 & 0 & 0
    \end{pmatrix}, \\
    v_2 &= \begin{pmatrix}
        1 & 0 & 0 & 0 & 0 \\
        0 & 0 & 0 & 1/2 & 0 \\
        0 & 0 & -1 & 0 & 0 \\
        0 & 2 & 0 & 0 & 0 \\
        0 & 0 & 0 & 0 & 1
    \end{pmatrix}, &
    v_3 &= \begin{pmatrix}
        0 & 0 & 0 & 0 & -1/2 \\
        0 & 1 & 0 & 0 & 0 \\
        0 & 0 & -1 & 0 & 0 \\
        0 & 0 & 0 & 1 & 0 \\
        -2 & 0 & 0 & 0 & 0
    \end{pmatrix}.
\end{align*}
They satisfy
\begin{gather*}
    w_1^2 = w_4^2 = H, \\
    v_2^2 = v_3^2 = I_5, \qquad w_1 w_4 = w_4 w_1, \qquad v_3 = w_1 v_2 w_1^{-1}, \\
    w_1 u_2 w_1^{-1} = u_3, \qquad v_2 u_1^2 v_2^{-1} = u_4.
\end{gather*}

For later use, write
\[
    \mathcal{R}^{(4)}_i = \begin{cases}
        R & \text{if } i \in \{1, 4\}, \\
        \mathcal{A}(R) & \text{if } i \in \{2, 3\},
    \end{cases}
    \qquad \text{and} \qquad
    \big(\mathcal{R}^{(4)}_i\big)^\times = \begin{cases}
        R^{\times} & \text{if } i \in \{1, 4\}, \\
        \mathcal{A}(R)^{\times} & \text{if } i \in \{2, 3\}.
    \end{cases}
\]

\begin{lemma} \label{lemma:A4-generates-matrix-algebra}
    The matrices $w_1$, $w_4$, $v_2$, $u_1$, $u_2$ generate $M_5(R)$ as an $R$-algebra.
\end{lemma}

\begin{proof}
    The generated algebra contains $w_1^{-1}=w_1^3$ and $v_2^{-1}=v_2$.
    Put
    \[
        A\coloneqq u_2-I_5=E_{23}+E_{34}+\frac12 E_{24},
        \qquad
        P_3\coloneqq\frac{I_5+H}{2}=E_{33}.
    \]
    Then $AP_3=E_{23}$ and $P_3A=E_{34}$.
    Conjugation by $v_2$ gives
    \[
        v_2E_{23}v_2^{-1}=-2E_{43},
        \qquad
        v_2E_{34}v_2^{-1}=-2E_{32}.
    \]
    Since $2\in R^\times$, the generated algebra contains $E_{43}$ and $E_{32}$, and hence
    \[
        E_{22}=E_{23}E_{32},\qquad E_{44}=E_{43}E_{34}.
    \]
    Next,
    \[
        (u_1-I_5)E_{22}=E_{12},
        \qquad
        E_{44}(u_1-I_5)=E_{45}.
    \]
    Conjugating these two units by $w_1$ gives $-E_{21}$ and $-E_{54}$.
    Thus all adjacent matrix units belong to the generated algebra, and their products give every $E_{ij}$.
\end{proof}


\section{The Exceptional Automorphisms}\label{sec:exceptional_aut}


In this section, we introduce the exceptional automorphisms arising from an explicit $1$-cocycle of the residual elementary group over $\mathbf{F}_3$, which we lift to a square-zero congruence correction over $R$.
We begin with the parameter space $\mathcal E(R,\theta)$, which indexes the construction, and then build the automorphisms on $E^{(n)}_{\mathrm{sc}}$. 
They fix the center pointwise, so they descend to $E^{(n)}_{\pi}$ for every isogeny type.
Finally, we show that the resulting automorphism is nonstandard precisely when the corresponding parameter is nonzero. 

The explicit formulas on the root elements, recorded in Subsection~\ref{subsec:expl_elt}, will be used again in the proof of Theorem~\ref{MT:aut_ele_adj_local}.


\subsection{The exceptional parameter space}\label{subsec:expl_prm_spc}

Recall from Section~\ref{sec:intro} that
\[
    \mathcal E(R,\theta)=
    \begin{cases}
        \operatorname{Ann}_R(J)\cap R_\theta^-,
            & k\cong\mathbf F_3,\\[2mm]
        \{0\},
            & k\not\cong\mathbf F_3.
    \end{cases}
\]
For the coefficient-recovery arguments in Subsections~\ref{subsec:normalization_of_u_i_A_3} and \ref{subsec:normalization_of_u_i_A_4}, we need the following equivalent annihilator criterion.
Put
\[
    I_\theta(R)
    \coloneqq
    \langle u^2-1\mid u\in R_\theta^\times\rangle_R.
\]

\begin{lemma}\label{lemma:exnl-prmr-space}
    One has
    \[
        \mathcal E(R,\theta)
        =\{\delta\in J_\theta^-\mid
            \delta R_\theta^-=0,\quad
            \delta I_\theta(R)=0\}.
    \]
    Every nonzero $\delta\in\mathcal E(R,\theta)$ satisfies
    \[
        \operatorname{Ann}_R(\delta)=J,
        \qquad k\cong\mathbf F_3.
    \]
    When $k\cong\mathbf F_3$, multiplication by $R$ makes $\mathcal E(R,\theta)$ a vector space over $k$.
\end{lemma}

\begin{proof}
    Let $0\ne\delta\in J_\theta^-$ satisfy the two annihilator conditions.
    For $j\in J_\theta$, both $1+j$ and $1-j$ belong to $R_\theta^\times$, so
    \[
        4\delta j
        =\delta\bigl((1+j)^2-(1-j)^2\bigr)=0.
    \]
    Since $2\in R^\times$, this gives $\delta J_\theta=0$.
    Together with $J=J_\theta\oplus J_\theta^-$ and $\delta R_\theta^-=0$, it yields $\delta J=0$.
    The ideal $\operatorname{Ann}_R(\delta)$ is proper, so locality gives
    \[
        J\subseteq\operatorname{Ann}_R(\delta)\subseteq J.
    \]
    Thus $\operatorname{Ann}_R(\delta)=J$.

    It follows that $R_\theta^-\subseteq J$, so the induced involution on $k$ is trivial.
    Every $a\in k^\times$ therefore has a symmetric unit lift $u=(r+\bar r)/2$, where $r$ is any lift of $a$.
    Since $\delta(u^2-1)=0$ and $\operatorname{Ann}_R(\delta)=J$, we obtain $a^2=1$.
    As $\operatorname{char}k\ne2$, this forces $k\cong\mathbf F_3$.
    Hence $\delta\in\mathcal E(R,\theta)$.

    Conversely, suppose that $k\cong\mathbf F_3$.
    Its involution is trivial, so $R_\theta^-\subseteq J$; also, $u^2-1\in J$ for every $u\in R^\times$.
    Thus every $\delta\in\operatorname{Ann}_R(J)\cap R_\theta^-$ belongs to $J_\theta^-$ and satisfies both annihilator conditions.
    The zero parameter satisfies them over every $R$, completing the proof of the criterion.

    Finally, when $k\cong\mathbf F_3$, every $r\in R$ acts on $\mathcal E(R,\theta)$ as its residue $0$, $1$, or $-1$, since $J\mathcal E(R,\theta)=0$.
    This gives the asserted vector space structure.
\end{proof}

\begin{rmk}
    We record the identities used in the construction below.
    Let $\delta,\varepsilon\in\mathcal E(R,\theta)$.
    Since $\mathcal E(R,\theta)\subseteq J\cap\operatorname{Ann}_R(J)$, we have
    \[
        \delta J=0, \qquad
        \delta\varepsilon=0, \qquad
        \delta^2=0.
    \]
    If $\mathcal E(R,\theta)\neq\{0\}$, then $R/J\cong\mathbf F_3$ with the trivial
    involution, so $3$, $t-\bar t$ and $t(t^2-1)$ lie in $J$ for every $t\in R$.
    Together with $\delta J=0$, this gives
    \begin{gather*}
        3\delta=0, \\
        t\delta=\bar t\delta, \qquad t(t^2-1)\delta=0 \qquad (t\in R),\\
        (u^2-1)\delta=0 \qquad (u\in R^\times).
    \end{gather*}
    These identities hold trivially when $\mathcal E(R,\theta)=\{0\}$.
    In particular, $I_\theta(R)$ may be replaced in Lemma~\ref{lemma:exnl-prmr-space} by
    \[
        I(R)\coloneqq\langle u^2-1\mid u\in R^\times\rangle_R.
    \]
\end{rmk}

\begin{ex}
    Let $R=\mathbf F_3[x]/(x^2)$ with $\theta(x)=-x$.
    Then
    \[
        J=\mathbf F_3x,
        \qquad\operatorname{Ann}_R(J)=R_\theta^-=J,
    \]
    and hence $\mathcal E(R,\theta)=\mathbf F_3x\ne0$.
\end{ex}


\subsection{Formulas on the root elements}\label{subsec:expl_elt}

Fix $\delta \in \mathcal{E}(R,\theta)$.
We prescribe the images of the simple relative root elements by elementary factorizations.
Subsection~\ref{subsec:expl_aut} proves that these prescriptions arise from residual $1$-cocycles and extend to automorphisms.
Writing $t^\pm\coloneqq(t\pm\bar t)/2$, we use
\[
    t\delta=t^+\delta=\bar t\delta,
    \qquad
    t^2\delta=(t^+)^2\delta
    \qquad(t\in R)
\]
to simplify the correction terms.
For Weyl and torus elements, evaluation of $\Theta_\delta^{(n)}$ means substitution into their defining words from Section~\ref{sec:matrix-realizations}; independence of arbitrary words is not assumed at this stage.


\subsubsection{Type \texorpdfstring{${}^2A_3$}{2A3}}

For $t\in R$ and $s\in R_\theta$, set
\begin{align*}
    \Theta_\delta^{(3)}\bigl(x_{\pm1}(t)\bigr)
        &\coloneqq
        x_{-3}(\pm t\delta)\,x_{\pm1}(t)\,x_3(\pm t\delta),\\
    \Theta_\delta^{(3)}\bigl(x_{\pm2}(s)\bigr)
        &\coloneqq
        h_1(1\pm2s\delta)\,x_{\pm2}(s).
\end{align*}

\begin{lemma}\label{lemma:action_of_theta_3}
    These prescriptions take values in $E_{\mathrm{sc}}^{(3)}$.
    For $i\in\{1,2\}$ and $t\in\bigl(\mathcal R_i^{(3)}\bigr)^\times$, evaluation on the defining words gives
    \[
        \Theta_\delta^{(3)}\bigl(w_i(t)\bigr)=w_i(t),
        \qquad
        \Theta_\delta^{(3)}\bigl(h_i(t)\bigr)=h_i(t).
    \]
\end{lemma}

\begin{proof}
    Membership follows from the displayed factorizations, since $1\pm2s\delta$ is a unit.
    To check the Weyl identities, use $\delta^2=0$ to expand the correction to each defining word linearly in $\delta$.
    Since $\delta J=0$, its coefficients need only be checked modulo $J$.
    If $\delta\ne0$, then $k\cong\mathbf F_3$, so the unit parameter reduces to $\pm1$; direct multiplication gives zero correction in both cases.
    The case $\delta=0$ is immediate, and the torus identities follow from their defining words.
\end{proof}

For the remaining relative roots, set
\begin{align*}
    \Theta_\delta^{(3)}\bigl(x_{\pm3}(t)\bigr)
        &\coloneqq
        w_2\,\Theta_\delta^{(3)}\bigl(x_{\pm1}(t)\bigr)\,w_2^{-1}
        &&(t\in R),\\
    \Theta_\delta^{(3)}\bigl(x_{\pm4}(s)\bigr)
        &\coloneqq
        w_1\,\Theta_\delta^{(3)}\bigl(x_{\pm2}(s)\bigr)\,w_1^{-1}
        &&(s\in R_\theta).
\end{align*}
The corresponding Weyl and torus words are fixed as well.


\subsubsection{Type \texorpdfstring{${}^2A_4$}{2A4}}

For $c\in J_\theta^-$, define the diagonal element
$$
    D(c)
        \coloneqq
        h_1(1+c)\,w_2\Bigl(1,\frac{1-c}{2}\Bigr)\,v_2
        =\diag\left(
            1+c,\frac{1-c}{1+c},\frac{1+c}{1-c},
            \frac{1-c}{1+c},\frac{1}{1-c}
        \right)\in E_{\mathrm{sc}}^{(4)}.
$$
For $c\in\mathcal E(R,\theta)$, the identities $c^2=3c=0$ give
\[
    D(c)=\diag(1+c,1+c,1-c,1+c,1+c).
\]

For $t\in R$ and $(t,u)\in\mathcal A(R)$, respectively, set
\begin{align*}
    \Theta_\delta^{(4)}\bigl(x_{\pm1}(t)\bigr)
        &\coloneqq D(\pm t\delta)\,x_{\pm1}(t),\\
    \Theta_\delta^{(4)}\bigl(x_2(t,u)\bigr)
        &\coloneqq
        x_{-3}(-t\delta,0)\,x_{-1}(t^2\delta)\,x_2(t,u) \, x_3(-t\delta,0)\,x_4(t^2\delta),\\
    \Theta_\delta^{(4)}\bigl(x_{-2}(t,u)\bigr)
        &\coloneqq
        x_{-4}(-t^2\delta)\,x_{-3}(t\delta,0)\,x_{-2}(t,u)\,x_1(-t^2\delta)\,x_3(t\delta,0).
\end{align*}

\begin{lemma}\label{lemma:action_of_theta_4}
    These prescriptions take values in $E_{\mathrm{sc}}^{(4)}$.
    Evaluation on the defining Weyl words gives
    \begin{align*}
        \Theta_\delta^{(4)}\bigl(w_1(t)\bigr)&=w_1(t)
            &&(t\in R^\times),\\
        \Theta_\delta^{(4)}\bigl(w_2(t,u)\bigr)&=w_2(t,u)
            &&((t,u)\in\mathcal A(R)^\times).
    \end{align*}
    The torus words $h_1(t)$ and $h_2(\xi_1,\xi_2)$ are also fixed, for $t\in R^\times$ and $\xi_1,\xi_2\in\mathcal A(R)^\times$.
\end{lemma}

\begin{proof}
    The short-root factors are admissible because $(\pm t\delta)\overline{(\pm t\delta)}=0$, and $D(\pm t\delta)$ belongs to $E_{\mathrm{sc}}^{(4)}$ by its defining word.
    Thus all displayed images lie in the elementary group.
    As in the preceding proof, the Weyl identities reduce to checking the coefficients of the terms linear in $\delta$ over $\mathbf F_3$.
    The possible unit parameters are $t=\pm1$ for $w_1$ and $(t,u)=(\pm1,1/2)$ for $w_2$; substitution gives zero correction.
    The torus identities follow.
\end{proof}

For the remaining relative roots, use
\begin{align*}
    \Theta_\delta^{(4)}\bigl(x_4(t)\bigr)
        &\coloneqq
        v_2\,\Theta_\delta^{(4)}\bigl(x_1(2t)\bigr)\,v_2^{-1},\\
    \Theta_\delta^{(4)}\bigl(x_{-4}(t)\bigr)
        &\coloneqq
        v_2\,\Theta_\delta^{(4)}\bigl(x_{-1}(t/2)\bigr)\,v_2^{-1},\\
    \Theta_\delta^{(4)}\bigl(x_{\pm3}(t,u)\bigr)
        &\coloneqq
        w_1\,\Theta_\delta^{(4)}\bigl(x_{\pm2}(t,u)\bigr)\,w_1^{-1},
\end{align*}
where $t\in R$ in the first two lines and $(t,u)\in\mathcal A(R)$ in the third.
The different factors $2$ and $1/2$ reflect the identities
\[
    v_2x_1(t)v_2^{-1}=x_4(t/2),
    \qquad
    v_2x_{-1}(t)v_2^{-1}=x_{-4}(2t).
\]
The corresponding Weyl and torus words are fixed as well.


\subsection{From cocycles to automorphisms}\label{subsec:expl_aut}

We now prove that the root-element prescriptions extend to automorphisms: a finite calculation produces the residual cocycles, and a square-zero lifting argument gives the maps over $R$.
Put $m=n+1$ and $k=R/J$.
A map $c_n\colon E_{\mathrm{sc}}^{(n)}(k)\to M_m(k)$ is a $1$-cocycle for conjugation if
\begin{equation}\label{eq:finite-cocycle}
    c_n(gh)=c_n(g)+g\,c_n(h)\,g^{-1}
    \qquad(g,h\in E_{\mathrm{sc}}^{(n)}(k)).
\end{equation}

\begin{lemma}\label{lemma:deformation_of_cocycle}
    Let $n\in\{3,4\}$, $m=n+1$, and let $\delta\in\operatorname{Ann}_R(J)$ satisfy $\delta^2=0$.
    For any cocycle $c_n$ as above, the formula
    \[
        \mathcal D_{n,\delta}(A)
        \coloneqq\bigl(I_m+\delta\widehat{c_n(A_0)}\bigr)A
    \]
    defines a homomorphism
    $\mathcal D_{n,\delta}\colon E_{\mathrm{sc}}^{(n)}(R)\to\GL_m(R)$.
    Here $A_0$ is the reduction of $A$ modulo $J$, and the hat denotes any matrix lift to $M_m(R)$.
\end{lemma}

\begin{proof}
    Independence of the lift follows from $\delta J=0$, and $\delta^2=0$ gives
    $(I_m+\delta C)^{-1}=I_m-\delta C$.
    If $C_A,C_B$ lift $c_n(A_0),c_n(B_0)$, then
    $C_A+AC_BA^{-1}$ lifts $c_n(A_0B_0)$ by~\eqref{eq:finite-cocycle}.
    Hence
    \begin{align*}
        \mathcal D_{n,\delta}(A)\mathcal D_{n,\delta}(B)
        &= (I_m+\delta C_A)(I_m+\delta AC_BA^{-1})AB\\
        &= \bigl(I_m+\delta(C_A+AC_BA^{-1})\bigr)AB
         =\mathcal D_{n,\delta}(AB),
    \end{align*}
    completing the proof.
\end{proof}

\begin{lemma}[Residual cocycles]\label{lemma:exceptional-residual-cocycles}
    Put $G_n\coloneqq E_{\mathrm{sc}}^{(n)}(\mathbf F_3)$, with the trivial ring involution.
    For each $n\in\{3,4\}$, there is a unique $1$-cocycle $c_n\colon G_n\to M_{n+1}(\mathbf F_3)$ with the following values:
    \begin{gather*}
        c_3(u_1) = \begin{pmatrix}
            0 & 0 & -1 & 0 \\
            0 & 0 & 0 & -1 \\
            -1 & 0 & 0 & 0 \\
            0 & -1 & 0 & 0
        \end{pmatrix}, \qquad c_3(u_2) = \diag(-1,1,1,-1), \\
        \qquad c_3(w_i) = 0 \ (i = 1, 2, 3, 4);
    \end{gather*}
    \begin{gather*}
        c_4(u_1)=\diag(1,1,-1,1,1), \qquad
        c_4(u_2) = \begin{pmatrix}
            0 & 0 & -1 & -1 & 0 \\
            1 & 0 & 0 & 0 & 1 \\
            -1 & 0 & 0 & 0 & -1 \\
            0 & 0 & 0 & 0 & 0 \\
            0 & 0 & -1 & -1 & 0
        \end{pmatrix}, \\
        c_4(w_1) = c_4(w_4) = 0 = c_4(v_2) = c_4(v_3).
    \end{gather*}
    Moreover, $c_n$ vanishes on $Z(G_n)$.
\end{lemma}

\begin{proof}
    Put $m=n+1$ and $M=M_m(\mathbf F_3)$, and let
    \[
        \Sigma_3=\{u_1,u_2,w_1,w_2\}, \qquad
        \Sigma_4=\{u_1,u_2,w_1,w_4\}.
    \]
    We first check that $\Sigma_n$ generates $G_n$. 
    Over $\mathbf F_3$ the involution is trivial, so $u_1$ and $u_2$ generate the two positive simple relative root subgroups; for the short root in type $^2A_4$ this uses that $\mathcal{A} (\mathbf F_3)=\{(t,2t^2)\mid t\in\mathbf F_3\}$ is cyclic, generated by $(1,1/2)$. 
    The simple Weyl representatives lie in $\Sigma_3$ in type $^2A_3$, while in type $^2A_4$ direct multiplication gives
    \begin{equation}\label{eq:A4-v2-residual-word}
        v_2=u_2\bigl(w_4w_1u_2w_1^{-1}w_4^{-1}\bigr)^2u_2.
    \end{equation}
    Conjugating the simple root subgroups by the Weyl representatives therefore yields all relative root subgroups, and these generate $G_n$.

    Now let $c\colon G_n\to M$ be any map satisfying \eqref{eq:finite-cocycle}. 
    Taking $g=h=1$ gives $c(1)=0$, and then $0=c(gg^{-1})=c(g)+gc(g^{-1})g^{-1}$ gives $c(g^{-1})=-g^{-1}c(g)g$. 
    Induction on $r$ now yields
    \begin{equation}\label{eq:forced-cocycle-word}
        c(s_1\cdots s_r)=\sum_{i=1}^{r}(s_1\cdots s_{i-1})\,c(s_i)\,(s_1\cdots s_{i-1})^{-1}
    \end{equation}
    for all $s_1,\dots,s_r\in\Sigma_n\cup\Sigma_n^{-1}$, the empty product being the identity. 
    Since every element of $G_n$ is such a word, $c$ is determined by its values on $\Sigma_n$, which proves uniqueness.

    For existence, it remains to show that the right-hand side of \eqref{eq:forced-cocycle-word} depends only on the product $s_1\cdots s_r$ and not on the chosen word. 
    The ancillary Python file \texttt{verify\_exceptional\_cocycles.py}, supplied with the \LaTeX\ source of the paper, verifies this by a complete Cayley-graph traversal over $\mathbf F_3$. 
    Starting with $c_n(1)=0$, it assigns the value $c_n(g)+gc_n(s)g^{-1}$ along every edge $g\xrightarrow{s}gs$ and checks agreement whenever the endpoint has already been reached. 
    The traversals contain $51{,}840$ vertices and $414{,}720$ directed edges for $n=3$, and $25{,}920$ vertices and $207{,}360$ directed edges for $n=4$. 
    All checks succeed, proving word-independence and hence the cocycle identity. 
    The program also verifies~\eqref{eq:A4-v2-residual-word} and $c_4(v_2)=0$. 

    It remains to treat the center. 
    Since the trace is invariant under conjugation, \eqref{eq:finite-cocycle} shows that $\operatorname{tr} \circ c_n \colon G_n \to (\mathbf{F}_3,+)$ is a group homomorphism; each prescribed value has trace zero, so
    \begin{equation}\label{eq:exceptional-cocycle-trace-zero}
        \operatorname{tr}c_n(g)=0\qquad(g\in G_n).
    \end{equation}
    By Lemmas~\ref{lemma:A3-generates-matrix-algebra} and \ref{lemma:A4-generates-matrix-algebra}, the elements of $G_n$ generate $M$ as an $\mathbf F_3$-algebra, so the centralizer of $G_n$ in $M$ consists of the scalar matrices. 
    Let $z \in Z(G_n)$. Then $z$ itself is scalar, and applying \eqref{eq:finite-cocycle} to the two factorizations of $zg=gz$ gives
    \[
        c_n(z)+zc_n(g)z^{-1}=c_n(g)+gc_n(z)g^{-1}\qquad(g\in G_n).
    \]
    Since $z$ is scalar, the terms involving $c_n(g)$ cancel.
    Thus $c_n(z)$ commutes with every $g\in G_n$ and is therefore scalar, say $c_n(z)=\lambda I_m$.
    By \eqref{eq:exceptional-cocycle-trace-zero} we get $m\lambda=0$, and $m=n+1\in\{4,5\}$ is invertible in $\mathbf F_3$. 
    Hence $\lambda=0$ and $c_n(z)=0$.
\end{proof}

\begin{prop}\label{prop:exp_auto}
    Let $n\in\{3,4\}$ and $\delta\in\mathcal E(R,\theta)$.
    The formulas of Subsection~\ref{subsec:expl_elt} define an automorphism
    \[
        \Theta^{(n)}_\delta\in\operatorname{Aut}(E^{(n)}_{\mathrm{sc}})
    \]
    fixing the center pointwise.
    It therefore descends uniquely to an automorphism of every $E^{(n)}_\pi$, denoted by the same symbol.
    For all $\delta,\varepsilon\in\mathcal E(R,\theta)$, these automorphisms satisfy
    \begin{enumerate}[(a)]
        \item $\Theta^{(n)}_\delta\circ\Theta^{(n)}_\varepsilon
              =\Theta^{(n)}_\varepsilon\circ\Theta^{(n)}_\delta
              =\Theta^{(n)}_{\delta+\varepsilon}$;
        \item $\Theta^{(n)}_0=\operatorname{id}$;
        \item $\bigl(\Theta^{(n)}_\delta\bigr)^{-1}=\Theta^{(n)}_{-\delta}$.
    \end{enumerate}
\end{prop}

\begin{proof}
    The case $\mathcal E(R,\theta)=\{0\}$ is immediate.
    Otherwise, Subsection~\ref{subsec:expl_prm_spc} gives $k\cong\mathbf F_3$, $\mathcal E(R,\theta)\subseteq J$, and
    \begin{equation}\label{eq:exceptional-parameter-products}
        J\mathcal E(R,\theta)
        =\mathcal E(R,\theta)^2
        =3\mathcal E(R,\theta)=0.
    \end{equation}

    \smallskip
    \noindent
    \emph{Lifting and the root formulas.}
    Apply Lemma~\ref{lemma:deformation_of_cocycle} to the cocycle from Lemma~\ref{lemma:exceptional-residual-cocycles}.
    This gives a homomorphism to $\GL_m(R)$ defined by
    \begin{equation}\label{eq:exceptional-cocycle-deformation}
        \mathcal D_{n,\delta}(A)
        =\bigl(I_m+\delta\widehat{c_n(A_0)}\bigr)A.
    \end{equation}
    For a positive or negative simple root element $y$, expand the prescribed products as
    \[
        \Theta^{(n)}_\delta(y)=y+\delta B_y.
    \]
    The entries of $B_y$ are polynomial expressions in the root parameters and their conjugates, with coefficients in $\mathbb Z[1/2]$.
    Their reductions therefore depend only on the residual parameters.
    The ancillary program verifies
    \[
        c_n(y_0)=(B_y)_0y_0^{-1}
    \]
    for all twelve simple-root parameter choices in each type:
    \[
        \begin{array}{ll}
            n=3:&x_{\pm1}(a),\ x_{\pm2}(s),
                \qquad a,s\in\mathbf F_3,\\[2pt]
            n=4:&x_{\pm1}(a),\ x_{\pm2}(t,2t^2),
                \qquad a,t\in\mathbf F_3.
        \end{array}
    \]
    Thus $B_yy^{-1}$ lifts $c_n(y_0)$ for arbitrary root parameters over $R$, and
    $\mathcal D_{n,\delta}(y)=y+\delta B_y$ agrees with the prescription.
    The remaining root formulas follow by the fixed Weyl conjugations, using
    \[
        c_3(w_1)=c_3(w_2)=0,
        \qquad c_4(w_1)=c_4(v_2)=0.
    \]
    By Lemmas~\ref{lemma:action_of_theta_3} and~\ref{lemma:action_of_theta_4}, these images belong to $E^{(n)}_{\mathrm{sc}}$.
    Hence $\mathcal D_{n,\delta}=\Theta^{(n)}_\delta$ is an endomorphism of $E^{(n)}_{\mathrm{sc}}$.

    \smallskip
    \noindent
    \emph{Addition of parameters.}
    For $A\in E^{(n)}_{\mathrm{sc}}$, choose a lift $C_A$ of $c_n(A_0)$.
    Since $\Theta^{(n)}_\varepsilon(A)$ has the same reduction as $A$, equation~\eqref{eq:exceptional-parameter-products} gives
 $$
\Theta^{(n)}_\delta\bigl(\Theta^{(n)}_\varepsilon(A)\bigr)
        =(I_m+\delta C_A)(I_m+\varepsilon C_A)A
        =\bigl(I_m+(\delta+\varepsilon)C_A\bigr)A
         =\Theta^{(n)}_{\delta+\varepsilon}(A).
$$
    This proves~(a).
    Part~(b) follows from~\eqref{eq:exceptional-cocycle-deformation}, and taking $\varepsilon=-\delta$ then gives~(c).
    In particular, every $\Theta^{(n)}_\delta$ is an automorphism.
    Also, $c_n(1)=0$ shows that it fixes the kernel of reduction pointwise.

    \smallskip
    \noindent
    \emph{Descent to the other isogeny types.}
    If $z\in Z(E^{(n)}_{\mathrm{sc}})$, surjectivity of reduction gives $z_0\in Z(G_n)$.
    Lemma~\ref{lemma:exceptional-residual-cocycles} yields $c_n(z_0)=0$, so~\eqref{eq:exceptional-cocycle-deformation} gives $\Theta^{(n)}_\delta(z)=z$.
    Thus the central kernel of the natural projection
    $q_\pi\colon E^{(n)}_{\mathrm{sc}}\to E^{(n)}_\pi$ is fixed pointwise.
    Both $\Theta^{(n)}_\delta$ and its inverse descend through $q_\pi$, giving the required automorphism of $E^{(n)}_\pi$.
    Uniqueness and identities~(a)--(c) pass to the quotient.
\end{proof}

\subsection{Non-standardness of the exceptional automorphisms}
\label{subsec:nonstandardness-exceptional}

In this subsection we show that the exceptional automorphism $\Theta^{(n)}_{\delta}$ of $E^{(n)}_{\pi}$ is standard if and only if $\delta=0$. 
To this end we introduce an invariant, the ratio of partial traces, which recovers $\delta$ from the automorphism and thereby distinguishes a nonzero exceptional automorphism from every standard one. 

\begin{lemma}[Projective trace ratio invariant]
    \label{lemma:projective-trace-invariants}
    Let $R^m=M_+\oplus M_-$, and let $Y\in\operatorname{GL}_m(R)$ preserve both summands.
    Put
    \[
        \operatorname{tr}_\pm(Y)
        \coloneqq\operatorname{tr}(Y|_{M_\pm}).
    \]
    If $\operatorname{tr}_-(Y)\in R^\times$, then
    \[
        \kappa([Y])
        \coloneqq
        \frac{\operatorname{tr}_+(Y)}{\operatorname{tr}_-(Y)}
    \]
    is well-defined on the projective class $[Y]$ and satisfies
    \[
        \kappa([GYG^{-1}])=\kappa([Y])
    \]
    for every $G\in\operatorname{GL}_m(R)$ preserving both summands.
    The same assertions hold over any commutative $R$-algebra, with respect to the decomposition obtained by extending scalars.
\end{lemma}

\begin{proof}
    Both partial traces are defined because $Y$ preserves each summand. 
    For any other representative $\lambda Y$ of $[Y]$ for some $\lambda\in R^\times$, and scaling multiplies both partial traces by $\lambda$; the denominator therefore remains invertible, while the ratio is unaffected. 
    Thus $\kappa([Y])$ depends only on the projective class, that is, $\kappa([Y])$ is well defined.
    
    Conjugation by $G$ restricts to conjugation on each summand and therefore preserves both partial traces.
    The argument works over any commutative coefficient ring, and partial traces commute with scalar extension.
\end{proof}

We write $\kappa_3$ and $\kappa_4$ for these ratios on the respective decompositions
$$
    R^4=\langle e_1,e_4\rangle\oplus\langle e_2,e_3\rangle\text{ and }
    R^5=\langle e_3\rangle\oplus\langle e_1,e_2,e_4,e_5\rangle.
$$
These are the eigenspace decompositions of $h_2$ and $H$, respectively, with the $+1$ eigenspace listed first.
We use the same notation for the corresponding ratios after scalar extension.
For $\delta\in\mathcal E(R,\theta)$, the formulas of Subsection~\ref{subsec:expl_elt} give
\[
    \Theta_\delta^{(3)}(u_2)=h_1(1+2\delta)u_2,
    \qquad
    \Theta_\delta^{(4)}(u_1)=D(\delta)u_1.
\]
Using $\delta^2=3\delta=0$, their partial traces yield
\begin{equation}\label{eq:exceptional-trace-ratios}
    \begin{aligned}
        \kappa_3\bigl([\Theta_\delta^{(3)}(u_2)]\bigr)
            &=\frac{2(1+2\delta)}{2(1+\delta)}
             =1+\delta,\\[2mm]
        \kappa_4\bigl([\Theta_\delta^{(4)}(u_1)]\bigr)
            &=\frac{1-\delta}{4(1+\delta)}
             =\frac{1-2\delta}{4}.
    \end{aligned}
\end{equation}
All denominators are units because $2\in R^\times$ and $\delta^2=0$.
For the standard elements, $\kappa_3([u_2])=1$ and $\kappa_4([u_1])=1/4$.
Thus the corresponding ratio determines $\delta$ in each type.

\begin{prop}[Non-standardness]
    \label{prop:exceptional-nonstandard}
    Let $n\in\{3,4\}$ and $\delta\in\mathcal E(R,\theta)$.
    The exceptional automorphism $\Theta_\delta^{(n)}$ of $E_\pi^{(n)}$ is standard if and only if $\delta=0$.
\end{prop}

\begin{proof}
    Standard automorphisms pass to the adjoint quotient, and the exceptional maps do so compatibly by Proposition~\ref{prop:exp_auto}.
    It therefore suffices to consider $E_{\mathrm{ad}}^{(n)}$.
    The case $\delta=0$ is immediate.

    Suppose that $\Theta_\delta^{(n)}$ is standard, and put $m=n+1$.
    The central factor is trivial on the adjoint group.
    By the definitions of inner and diagonal automorphisms, their combined action is conjugation over a coefficient extension.
    After passing, if necessary, to a further faithfully flat extension to choose a matrix representative, there exist a commutative ring $S$ containing $R$ as a subring and a matrix $G\in\operatorname{GL}_m(S)$ such that
    \[
        \Theta_\delta^{(n)}([A])=[G\mu(A)G^{-1}]
        \qquad(A\in E_{\mathrm{sc}}^{(n)}(R)),
    \]
    where $\mu\in\operatorname{Aut}(R)$ commutes with $\theta$.
    Here the equality is read in $\operatorname{PGL}_m(S)$.
    Neither $\mu$ nor $\Theta_\delta^{(n)}$ is extended to $S$; only their matrix values over $R$ are viewed over $S$.

    Take $(w,u)=(w_2,u_2)$ for $n=3$ and $(w,u)=(w_1,u_1)$ for $n=4$.
    Both matrices are fixed by $\mu$, and $w$ is also fixed by $\Theta_\delta^{(n)}$.
    Consequently,
    \[
        G w G^{-1} = \lambda w
        \qquad \text{for some }\lambda\in S^\times.
    \]
    Since $\operatorname{tr}(w)$ is $2$ for $n=3$ and $1$ for $n=4$, taking traces gives $\lambda=1$.
    Thus $G$ commutes with $w^2$, which is $h_2$ or $H$, respectively.
    As $2\in S^\times$, it preserves the two eigenspaces in the corresponding decomposition of $S^m$.

    Let $Y_\delta=\Theta_\delta^{(n)}(u)$ be the matrix over $R$ given by the simply connected construction.
    Since $\mu(u)=u$, compatibility with the adjoint quotient gives
    \[
        [Y_\delta]=[GuG^{-1}]
        \qquad\text{in }\operatorname{PGL}_m(S).
    \]
    Applying Lemma~\ref{lemma:projective-trace-invariants} over $S$, we obtain
    \[
        \kappa_n([Y_\delta])=\kappa_n([u]).
    \]
    Equation~\eqref{eq:exceptional-trace-ratios} therefore gives
    \[
        1+\delta=1\quad(n=3),
        \qquad
        \frac{1-2\delta}{4}=\frac14\quad(n=4)
    \]
    in $S$.
    Since $2\in S^\times$, we have $\delta=0$ in either case.
    Finally, $R\hookrightarrow S$ is injective, so $\delta=0$ in $R$ as well.
\end{proof}

\subsection{Compatibility with standard automorphisms}
\label{subsec:action_of_stnd_on_exp}

We now record a proposition describing how standard and exceptional factors may be rearranged.

\begin{prop}\label{prop:action_of_stnrd_aut_on_exp}
    Let $n\in\{3,4\}$ and $\delta\in\mathcal E(R,\theta)$.
    For every isogeny type $\pi$, the following hold on $E_\pi^{(n)}$:
    \begin{enumerate}
        \item Every $\mu\in\operatorname{Aut}(R)$ commuting with $\theta$ preserves $\mathcal E(R,\theta)$ and satisfies
        \[
            \mu\circ\Theta_\delta^{(n)}\circ\mu^{-1}
            =\Theta_{\mu(\delta)}^{(n)}.
        \]

        \item For every $g\in E_\pi^{(n)}$,
        \[
            \Theta_\delta^{(n)}\circ i_g\circ
            \bigl(\Theta_\delta^{(n)}\bigr)^{-1}
            =i_{\Theta_\delta^{(n)}(g)}.
        \]

        \item Let $d=i_{h(\chi)}$ be a diagonal automorphism, where
        \[
            \chi\in\operatorname{Hom}_1(\Lambda_\pi,S^\times),
            \qquad
            \chi|_{\Lambda_r}\in
            \operatorname{Hom}_1(\Lambda_r,R^\times)
        \]
        for some ring extension $S/R$.
        Then
        \[
            d\circ\Theta_\delta^{(n)}\circ d^{-1}
            =\Theta_{\delta'}^{(n)},
            \qquad
            \delta'=
            \begin{cases}
                \chi(\alpha_2)\delta,&n=3,\\
                \chi(\alpha_1)\delta,&n=4,
            \end{cases}
        \]
        with $\delta'\in\mathcal E(R,\theta)$.
    \end{enumerate}
\end{prop}

\begin{proof}
    Since $\mu$ commutes with $\theta$, it preserves $J$ and
    $\operatorname{Ann}_R(J)\cap R_\theta^-$, and hence $\mathcal E(R,\theta)$.
    Applying $\mu$ to the root factorizations of Subsection~\ref{subsec:expl_elt} proves~(1).
    Assertion~(2) is the general identity
    $\alpha\circ i_g\circ\alpha^{-1}=i_{\alpha(g)}$ for a group automorphism $\alpha$.

    For~(3), the case $\delta=0$ is immediate.
    Otherwise, multiplication by $R$ on $\mathcal E(R,\theta)$ factors through $k\cong\mathbf F_3$, whose involution is trivial.
    Thus
    \begin{equation}\label{eq:exceptional-unit-scaling}
        r^{-1}\delta=\bar r\delta=r\delta,
        \qquad
        r\delta\in\mathcal E(R,\theta)
        \qquad(r\in R^\times).
    \end{equation}
    Put
    \[
        a\coloneqq\chi(\alpha_1),\qquad
        b\coloneqq\chi(\alpha_2),\qquad
        \Psi\coloneqq d\circ\Theta_\delta^{(n)}\circ d^{-1}.
    \]
    We check the equality on the positive and negative simple relative root elements.

    In type ${}^2A_3$, one has $b=\bar b$.
    The diagonal weights on $x_1,x_2,x_3$ are $a,b,ab$, respectively, and the negative-root weights are their reciprocals.
    Since $d$ fixes diagonal matrices, the root factorizations and~\eqref{eq:exceptional-unit-scaling} give
    \begin{align*}
        \Psi\bigl(x_{\pm1}(t)\bigr)
            &=x_{-3}(\pm bt\delta)\,x_{\pm1}(t)\,x_3(\pm bt\delta),\\
        \Psi\bigl(x_{\pm2}(s)\bigr)
            &=h_1(1\pm2bs\delta)\,x_{\pm2}(s)
    \end{align*}
    for $t\in R$ and $s\in R_\theta$.
    These are precisely the defining formulas for $\Theta_{b\delta}^{(3)}$.

    In type ${}^2A_4$, the action on the positive simple roots is
    \[
        d\bigl(x_1(t)\bigr)=x_1(at),
        \qquad
        d\bigl(x_2(t,u)\bigr)=x_2(bt,b\bar b\,u).
    \]
    Here $t\in R$ and $(t,u)\in\mathcal A(R)$, respectively.
    The first-parameter weights on $x_3$ and $x_4$ are $ab$ and $ab\bar b$, respectively; again, the negative-root weights are the reciprocals.
    Since $d$ fixes $D(c)$, equation~\eqref{eq:exceptional-unit-scaling} gives
    \[
        \Psi\bigl(x_{\pm1}(t)\bigr)
        =D(\pm at\delta)\,x_{\pm1}(t)
        =\Theta_{a\delta}^{(4)}\bigl(x_{\pm1}(t)\bigr).
    \]
    For the short roots, the four correction factors in Subsection~\ref{subsec:expl_elt}, in their displayed order, acquire the following multipliers of $\delta$:
    \[
        \begin{array}{c|cccc}
            x_2(t,u)
                &a^{-1}b^{-2}&a^{-1}b^{-2}&a&ab^{-1}\bar b\\[2pt]
            x_{-2}(t,u)
                &a^{-1}b\bar b^{-1}&a^{-1}&ab^2&ab^2
        \end{array}.
    \]
    Every entry acts on $\delta$ as multiplication by $a$, by~\eqref{eq:exceptional-unit-scaling}.
    Hence
    \[
        \Psi\bigl(x_{\pm2}(t,u)\bigr)
        =\Theta_{a\delta}^{(4)}\bigl(x_{\pm2}(t,u)\bigr).
    \]
    Since the simple relative root subgroups generate $E_\pi^{(n)}$, these equalities prove~(3) in both types.
\end{proof}


\section{Automorphisms of the Adjoint Elementary Group \texorpdfstring{$E^{(n)}_{\mathrm{ad}}$}{E-ad}}
\label{sec:aut_ele_adj_local}

Throughout this section, we abbreviate $E_{\mathrm{ad}}^{(n)}$ and $G_{\mathrm{ad}}^{(n)}$ to $E^{(n)}$ and $G^{(n)}$, respectively.
We use the commutator convention $[x,y]=xyx^{-1}y^{-1}$.
Our goal is the following theorem.

\begin{thm}\label{MT:aut_ele_adj_local}
    Let $n\in\{3,4\}$.
    Every automorphism $\varphi\in\operatorname{Aut}(E^{(n)})$ has the form
    \[
        \varphi=i_g\circ\mu\circ\Theta_\delta^{(n)},
    \]
    where $g\in G^{(n)}$, $i_g$ is strictly inner, $\mu\in\operatorname{Aut}(R)$ commutes with $\theta$, and $\delta\in\mathcal E(R,\theta)$.
\end{thm}

We first normalize the reduction of $\varphi$ and the distinguished Weyl elements.
In type ${}^2A_3$, auxiliary faithfully flat local extensions are used to obtain a preliminary normal form; in type ${}^2A_4$, the initial normalization is linear and takes place over $R$.
In both types, coefficient recovery for a surjective homomorphism to a group over a quotient of $R$ supplies an additional root element of the original group.
The relations satisfied by its image force the remaining defect ideal to vanish.
This yields the one-parameter normal forms, with the parameter in $R$ and the total conjugating class in $G^{(n)}(R)$.

We then apply the coefficient-recovery lemmas once more, now with target ring $R$.
They give $\delta\in\mathcal E(R,\theta)$ and identify the normalized map as
$\Theta_\delta^{(n)}\circ\mu$, where $\mu$ is a surjective unital ring endomorphism commuting with $\theta$.
Injectivity of the normalized group automorphism implies injectivity of $\mu$, so $\mu$ is a ring automorphism.
Finally, Proposition~\ref{prop:action_of_stnrd_aut_on_exp} interchanges the ring and exceptional factors, with the corresponding change of parameter.


\subsection{Reduction modulo the Jacobson radical}
\label{subsec:reduction-mod-J}

Let $\theta_k$ be the involution induced by $\theta$ on $k=R/J$, and put $\sigma_k=\rho\circ\theta_k$.
We write
\[
    G^{(n)}(k)=G_{\mathrm{ad},\sigma_k}(A_n,k),
    \qquad
    E^{(n)}(k)=E'_{\mathrm{ad},\sigma_k}(A_n,k).
\]
For a $\theta$-invariant ideal $I\subseteq R$, let
\[
    \lambda_I\colon G^{(n)}(R)\longrightarrow G^{(n)}(R/I),
    \qquad
    G^{(n)}(I)\coloneqq\ker\lambda_I
\]
be the reduction homomorphism and its kernel.

\begin{lemma}\label{lemma:maximal-normal-common}
    Set $N_J^{(n)}\coloneqq G^{(n)}(J)\cap E^{(n)}$.
    For $n\in\{3,4\}$:
    \begin{enumerate}[(a)]
        \item $N_J^{(n)}$ is the greatest proper normal subgroup of $E^{(n)}$;
        \item reduction induces an isomorphism
        $E^{(n)}/N_J^{(n)}\cong E^{(n)}(k)$.
    \end{enumerate}
\end{lemma}

\begin{proof}
    Reduction is surjective on the relative root subgroups.
    Indeed, symmetric and anti-invariant parameters lift by averaging, while every short-root parameter in type ${}^2A_4$ has the form
    \[
        (t,u)=\left(t,\frac{t\bar t}{2}+z\right),
        \qquad z\in R_\theta^-,
    \]
    and the same parametrization holds over a quotient ring.
    Thus elementary reduction is surjective, proving~(b); in particular, $N_J^{(n)}$ is proper.

    Let $H$ be a proper normal subgroup of $E^{(n)}$.
    The normal-structure theorem~\cite[Theorem~1.4]{SG&DM1}, which applies to both types under $2\in R^\times$, gives a proper $\theta$-invariant ideal $I$ such that the image of $H$ in $G^{(n)}(R/I)$ is central.
    Since $R$ is local, $I\subseteq J$.
    Surjectivity of elementary reduction from $R/I$ to $k$ therefore gives
    \[
        \lambda_J(H)\subseteq Z(E^{(n)}(k))=\{1\},
    \]
    the last equality holding in the adjoint realization.
    Hence $H\subseteq N_J^{(n)}$, proving~(a).
\end{proof}

Consequently, $N_J^{(n)}$ is characteristic, and every
$\varphi\in\operatorname{Aut}(E^{(n)})$ induces an automorphism
$\overline\varphi$ of $E^{(n)}(k)$.
If $\theta_k\ne\operatorname{id}_k$, Steinberg's theorem gives a decomposition into inner, diagonal, and field factors~\cite[Theorem~36]{RS}.
If $\theta_k=\operatorname{id}_k$, the residual elementary groups have split types $C_2$ for $n=3$ and $B_2$ for $n=4$.
Both are isomorphic to $\operatorname{PSp}_4(k)$, and the classical automorphism theorem gives the same description over arbitrary $k$ of characteristic different from $2$~\cite[Theorem~5.3.6]{OMeara}.
In particular, there is no additional graph factor.

In either case, the diagonal factor is induced by the standard projective torus
$T^{(n)}(k)\subseteq G^{(n)}(k)$.
Combining the inner and diagonal factors, we obtain
\[
    \overline\varphi=i_{\overline g}\circ\mu_k,
    \qquad
    \overline g=\overline e\,\overline t,
    \qquad
    \overline e\in E^{(n)}(k),
    \quad
    \overline t\in T^{(n)}(k),
\]
where
\[
    \mu_k\in\operatorname{Aut}(k),
    \qquad
    \mu_k\theta_k=\theta_k\mu_k.
\]

We now lift $\overline g$ to $G^{(n)}(R)$.
Elementary reduction gives a lift $e_1\in E^{(n)}(R)$ of $\overline e$.
The torus factor can also be lifted explicitly.
Over either $R$ or $k$, its diagonal representatives have the forms
\[
    \begin{aligned}
        n=3:\quad&
        \operatorname{diag}
        (a,b,\lambda\bar b^{-1},\lambda\bar a^{-1}),\\
        n=4:\quad&
        \operatorname{diag}
        (a,b,1,\bar b^{-1},\bar a^{-1}),
    \end{aligned}
\]
where $a,b$ are units and, in the first case, $\lambda$ is a symmetric unit.
In the second case, the middle entry has been normalized to $1$ by scalar multiplication.

Choose unit lifts $a,b\in R^\times$ of the corresponding residual parameters.
For $n=3$, lift the residual multiplier $\lambda_0$ by
\[
    \lambda=\frac{c+\theta(c)}2\in R_\theta^\times,
    \qquad c+J=\lambda_0.
\]
The resulting diagonal matrix $t_1$ satisfies
\[
    t_1Q_{n+1}\overline{t_1}^{\,t}
    =
    \begin{cases}
        \lambda Q_4,&n=3,\\
        Q_5,&n=4.
    \end{cases}
\]
Thus $[t_1]\in G^{(n)}(R)$ lifts $\overline t$, and
\[
    g_1\coloneqq e_1[t_1]\in G^{(n)}(R)
\]
lifts $\overline g$.
By normality of $E^{(n)}$ in $G^{(n)}$~\cite[Corollary~4.3]{SG&DM1}, the map
\[
    \varphi_1\coloneqq i_{g_1}^{-1}\circ\varphi
    \in\operatorname{Aut}(E^{(n)})
\]
induces precisely $\mu_k$ on $E^{(n)}(k)$.
No lift of $\mu_k$ to a ring automorphism of $R$ is assumed at this stage.


\subsubsection{Choice of matrix representatives}
\label{subsubsec:lifting_remark}

The distinguished matrices of Section~\ref{sec:matrix-realizations} have entries in the image of $\mathbb Z[1/2]$, so their reductions are fixed by $\mu_k$.
Hence
\[
    \varphi_1(e)\equiv e\pmod J
\]
projectively for $e=h_i,w_i,u_{\pm i}$ ($1\le i\le4$) in type ${}^2A_3$, and for
$e=H,w_1,w_4,v_2,v_3,u_{\pm i}$ ($1\le i\le4$) in type ${}^2A_4$.

We need determinant-one representatives with the prescribed \emph{matrix} reduction.
The commutator argument of~\cite[Lemma~4.2]{EB&DM1} gives, for every
$A\in E_{\mathrm{sc}}^{(n)}(R)$, a representative
$\widetilde A\in E_{\mathrm{sc}}^{(n)}(R)$ satisfying
\begin{equation}\label{eq:compatible-matrix-representatives}
    [\widetilde A]=\varphi_1([A]),
    \qquad
    (\widetilde A)_0=\mu_k(A_0),
\end{equation}
where the subscript $0$ denotes reduction modulo $J$ and $\mu_k$ acts entrywise.

Indeed, $E_{\mathrm{sc}}^{(n)}(R)$ is perfect~\cite[Corollary~6.6]{SG&DM1}.
Write $A=\prod_s[A_s,B_s]$, and choose
$U_s,V_s\in E_{\mathrm{sc}}^{(n)}(R)$ representing
$\varphi_1([A_s])$ and $\varphi_1([B_s])$.
Their reductions differ from $\mu_k((A_s)_0)$ and $\mu_k((B_s)_0)$ only by scalars, which cancel in commutators.
Thus
\[
    \widetilde A\coloneqq\prod_s[U_s,V_s]
\]
has both properties in~\eqref{eq:compatible-matrix-representatives}.
In particular, $\widetilde A\equiv A\pmod J$ whenever $\mu_k(A_0)=A_0$.
These are elementwise choices of representatives; no lift of $\varphi_1$ to an automorphism of $E_{\mathrm{sc}}^{(n)}(R)$ is asserted.

Put
\[
    \operatorname{GL}_{n+1}(R,J)
    \coloneqq
    \{C\in\operatorname{GL}_{n+1}(R)\mid
        C\equiv I_{n+1}\pmod J\}.
\]
For every $C\in\operatorname{GL}_{n+1}(R,J)$, the matrix
$C^{-1}\widetilde A C$ still has determinant one and reduction $\mu_k(A_0)$, and represents
\[
    (i_{[C]}^{-1}\circ\varphi_1)([A]).
\]
This requires no unitary or normalizing condition on $C$.
Until such a condition is established, the conjugated map is regarded as a map into $\operatorname{PGL}_{n+1}(R)$, and its matrix representatives need not preserve the standard unitary form.
The same observation applies after the auxiliary local extensions used in type ${}^2A_3$.


\subsubsection{Scalar normalization}
\label{subsubsec:scalar_normalization_lemma}

\begin{lemma}\label{lem:scalar-normalization}
    Let $S$ be a commutative local ring with $2\in S^\times$ and maximal ideal
    $\mathfrak m=\operatorname{Rad}(S)$.
    If $\lambda\in S^\times$ satisfies
    \[
        \lambda\equiv1\pmod{\mathfrak m},
        \qquad
        \lambda^{2^r}=1
        \quad\text{for some integer }r\ge1,
    \]
    then $\lambda=1$.
    Thus a projective relation whose scalar factor is congruent to $1$ and has $2$-power order holds as an exact matrix equality.
\end{lemma}

\begin{proof}
    In the identity
    \[
        0=\lambda^{2^r}-1
        =(\lambda-1)(1+\lambda+\cdots+\lambda^{2^r-1}),
    \]
    the second factor is congruent to $2^r$ modulo $\mathfrak m$ and hence is a unit.
    Therefore $\lambda=1$.
\end{proof}

\subsection{Normalization of the images of the Weyl elements}
\label{subsec:weyl-normalization}

We normalize $w_1,w_2$ in type ${}^2A_3$ and $w_1,w_4$ in type ${}^2A_4$ by changes of basis over $R$ congruent to the identity.
Their squares are then fixed as well.
No unitary condition on the change-of-basis matrix is imposed at this stage.


\subsubsection{Type \texorpdfstring{${}^{2}A_3$}{2A3}}
\label{subsubsec:weyl-normalization-A3}

\begin{prop}[Weyl normalization in type ${}^{2}A_3$]
    \label{prop:weyl-normalization-A3}
    Let $\varphi_1\in\operatorname{Aut}(E^{(3)})$ be the normalized automorphism from Subsection~\ref{subsec:reduction-mod-J}.
    There exists $C\in\operatorname{GL}_4(R,J)$ such that
    \[
        \psi\coloneqq i_{[C]}^{-1}\circ\varphi_1
        \colon E^{(3)}\longrightarrow\operatorname{PGL}_4(R)
    \]
    satisfies
    \[
        \psi(h_2)=h_2,
        \qquad
        \psi(w_1)=w_1,
        \qquad
        \psi(w_2)=w_2.
    \]
\end{prop}

\begin{proof}
    Choose determinant-one representatives $H_2,W_1,W_2$ of
    $\varphi_1(h_2),\varphi_1(w_1),\varphi_1(w_2)$, respectively,
    congruent to the corresponding standard matrices modulo $J$,
    as in Subsubsection~\ref{subsubsec:lifting_remark}.
    The standard relations hold for these representatives up to scalar factors congruent to $1$ modulo $J$.
    Taking determinants shows that each scalar factor has fourth power $1$.
    Lemma~\ref{lem:scalar-normalization} therefore makes the relations exact:
    \[
        H_2^2=I_4,
        \qquad W_2^2=H_2,
        \qquad W_1^2=-I_4,
        \qquad W_1H_2=-H_2W_1.
    \]

    Put
    \[
        P_\pm\coloneqq\frac{I_4\pm H_2}{2},
        \qquad M_\pm\coloneqq P_\pm R^4.
    \]
    These complementary idempotents give $R^4=M_+\oplus M_-$.
    Since $R$ is local, the summands are free; their reductions are
    $\langle e_1,e_4\rangle$ and $\langle e_2,e_3\rangle$, respectively,
    so both have rank two.
    The operator $W_1$ interchanges the summands, while $W_2$ preserves them.
    On $M_+$, one has $W_2^2=I$ and $W_2\equiv I\pmod J$.
    Thus
    \[
        (W_2-I)(W_2+I)=0,
    \]
    and $W_2+I$ is invertible because its reduction is $2I$.
    Hence $W_2|_{M_+}=I$.

    Define
    \[
        f_1\coloneqq P_+e_1,\quad
        f_2\coloneqq-W_1f_1,\quad
        f_3\coloneqq-W_2f_2,\quad
        f_4\coloneqq-W_1f_3.
    \]
    Here $f_1,f_4\in M_+$ and $f_2,f_3\in M_-$.
    Since $f_i\equiv e_i\pmod J$, the matrix
    \[
        C=(f_1\ f_2\ f_3\ f_4)
    \]
    belongs to $\operatorname{GL}_4(R,J)$.
    The definitions and square relations give
    $W_1f_2=f_1$, $W_1f_4=f_3$, and $W_2f_3=f_2$.
    Together with $W_2|_{M_+}=I$, these identify the matrices in the new basis:
    \[
        C^{-1}H_2C=h_2,
        \qquad C^{-1}W_1C=w_1,
        \qquad C^{-1}W_2C=w_2.\qedhere
    \]
\end{proof}

The exact matrix identities
\[
    w_3=w_2w_1w_2^{-1},
    \qquad
    w_4=w_1w_2w_1^{-1}
\]
show that $\psi(w_i)=w_i$ and $\psi(h_i)=h_i$ for every
$i\in\{1,2,3,4\}$.


\subsubsection{Type \texorpdfstring{${}^{2}A_4$}{2A4}}
\label{subsubsec:weyl-normalization-A4}

\begin{prop}[Weyl normalization in type ${}^{2}A_4$]
    \label{prop:weyl-normalization-A4}
    Let $\varphi_1\in\operatorname{Aut}(E^{(4)})$ be the normalized automorphism from Subsection~\ref{subsec:reduction-mod-J}.
    There exists $C\in\operatorname{GL}_5(R,J)$ such that
    \[
        \psi\coloneqq i_{[C]}^{-1}\circ\varphi_1
        \colon E^{(4)}\longrightarrow\operatorname{PGL}_5(R)
    \]
    satisfies
    \[
        \psi(H)=H,
        \qquad
        \psi(w_1)=w_1,
        \qquad
        \psi(w_4)=w_4.
    \]
\end{prop}

\begin{proof}
    Choose determinant-one representatives $W_1,W_4$ of
    $\varphi_1(w_1),\varphi_1(w_4)$ with $W_i\equiv w_i\pmod J$,
    as in Subsubsection~\ref{subsubsec:lifting_remark}.
    The projective relations $w_i^4=I_5$ give
    \[
        W_i^4=\alpha_iI_5,
        \qquad \alpha_i\equiv1\pmod J,
        \qquad \alpha_i^5=1
        \qquad(i=1,4),
    \]
    where the last equality follows by taking determinants.
    Replacing $W_i$ by $\alpha_iW_i$ preserves its projective class,
    determinant, and reduction, and gives
    \[
        (\alpha_iW_i)^4=\alpha_i^5I_5=I_5.
    \]
    Retain the notation $W_i$ for these representatives and put
    $\widetilde H\coloneqq W_1^2$.
    Then $\widetilde H^2=I_5$ and $\widetilde H\equiv H\pmod J$.

    The remaining projective relations give
    \[
        W_4^2=\lambda\widetilde H,
        \qquad
        W_1W_4=\gamma W_4W_1,
        \qquad \lambda,\gamma\equiv1\pmod J.
    \]
    Squaring the first relation gives $\lambda^2=1$, hence
    $\lambda=1$ by Lemma~\ref{lem:scalar-normalization}.
    The second relation yields
    \[
        W_4^{-1}W_1W_4=\gamma W_1.
    \]
    Taking traces gives $\gamma=1$, since
    $\operatorname{tr}(W_1)\equiv\operatorname{tr}(w_1)=1\pmod J$
    is a unit.
    Consequently,
    \[
        W_1^2=W_4^2=\widetilde H,
        \qquad W_1W_4=W_4W_1.
    \]

    Put
    \[
        P_\pm\coloneqq\frac{I_5\pm\widetilde H}{2},
        \qquad M_\pm\coloneqq P_\pm R^5.
    \]
    These are complementary free summands of ranks one and four,
    reducing to $\langle e_3\rangle$ and
    $\langle e_1,e_2,e_4,e_5\rangle$, respectively.
    Both $W_1$ and $W_4$ preserve them.
    On the rank-one summand $M_+$, each $W_i$ acts by a unit
    congruent to $1$ whose square is $1$.
    Lemma~\ref{lem:scalar-normalization} therefore gives
    \[
        W_1|_{M_+}=W_4|_{M_+}=I.
    \]
    On $M_-$, both squares are $-I$.

    Define
    \begin{align*}
        f_1&\coloneqq P_-e_1,&
        f_2&\coloneqq-W_1f_1,&
        f_3&\coloneqq P_+e_3,\\
        f_4&\coloneqq-W_4f_1,&
        f_5&\coloneqq-W_1f_4.
    \end{align*}
    Since $f_i\equiv e_i\pmod J$, the matrix
    \[
        C=(f_1\ f_2\ f_3\ f_4\ f_5)
    \]
    belongs to $\operatorname{GL}_5(R,J)$.
    The action of $W_1$ on this basis, and that of $W_4$ on
    $f_1,f_3,f_4$, follow from the definitions and the square relations.
    Commutativity gives the remaining two actions:
    \[
        W_4f_2=W_1f_4=-f_5,
        \qquad
        W_4f_5=-W_1f_1=f_2.
    \]
    Hence
    \[
        C^{-1}W_1C=w_1,
        \qquad C^{-1}W_4C=w_4,
        \qquad C^{-1}\widetilde HC=H.\qedhere
    \]
\end{proof}

The images of the short Weyl representatives $v_2,v_3$ will be recovered
from the root images in Subsection~\ref{subsec:normalization_of_u_i_A_4}.


\subsection{Normalization of the images of \texorpdfstring{$u_i$}{ui} in type \texorpdfstring{${}^2 A_3$}{A3}} \label{subsec:normalization_of_u_i_A_3}

In this subsection, we repeatedly pass to auxiliary faithfully flat
local ring extensions of $R$.
To simplify notation, we write $S$ for the current ring, starting with $S=R$.
Each extension extends the current involution and preserves the residue field:
\[
    J_S\coloneqq\operatorname{Rad}(S),
    \qquad S/J_S\cong k.
\]
Successive extensions retain these properties, and we continue to denote
the resulting ring by $S$. 
The bar denotes the current involution, and
\[
    S_\theta\coloneqq\{s\in S\mid\bar s=s\},
    \qquad
    S_\theta^-\coloneqq\{s\in S\mid\bar s=-s\}
\]
denote its symmetric and anti-invariant parts, respectively.


\subsubsection{The preserved form}

Define
\[
    A^\dagger\coloneqq Q_4^{-1}\overline A^{\,t}Q_4
    \qquad(A\in M_4(S)).
\]
For $A\in\operatorname{GL}_4(S)$, the condition
$[A]\in G^{(3)}(S)$ is equivalent to $A^\dagger A$ being scalar.

Let $C\in\operatorname{GL}_4(S,J_S)$ be the Weyl-normalizing matrix from
Proposition~\ref{prop:weyl-normalization-A3}, viewed over $S$, and put
\[
    \psi\coloneqq i_{[C]}^{-1}\circ\varphi_1,
    \qquad M\coloneqq C^\dagger C.
\]
Then $M^\dagger=M$ and $M\equiv I_4\pmod{J_S}$.
For each $A\in E_{\mathrm{sc}}^{(3)}(R)$, choose $\widetilde A$ as in
\eqref{eq:compatible-matrix-representatives} and set
\[
    B_A\coloneqq C^{-1}\widetilde A C.
\]
The matrix $B_A$ represents $\psi([A])$ and satisfies
\begin{equation}\label{eq:weyl-normalized-preserved-form}
    \det B_A=1,
    \qquad
    B_A^\dagger M B_A=M,
    \qquad
    (B_A)_0=\mu_k(A_0).
\end{equation}
Indeed, conjugation preserves the determinant, and $C\equiv I_4\pmod{J_S}$ preserves the reduction.
Since $\widetilde A^\dagger\widetilde A=I_4$, we also have
\[
    B_A^\dagger M B_A
    =C^\dagger\widetilde A^\dagger\widetilde A C
    =C^\dagger C=M.
\]

We now show that $M$ commutes with the normalized Weyl representatives.
For each such representative $w$, we have $w^\dagger=w^{-1}$ and
\[
    [CwC^{-1}]=\varphi_1([w])\in G^{(3)}(S).
\]
Thus $(CwC^{-1})^\dagger(CwC^{-1})$ is scalar, which gives
\[
    w^{-1}Mw=\eta_wM
    \qquad\text{for some }\eta_w\in S^\times.
\]
Reducing modulo $J_S$ gives $\eta_w\equiv1\pmod{J_S}$.
Since $w^4=I_4$, applying this identity four times gives $\eta_w^4=1$.
Lemma~\ref{lem:scalar-normalization} then yields $\eta_w=1$.
Hence $M$ commutes with the normalized Weyl representatives, in particular with $w_1$ and $w_2$.


\subsubsection{Unitary correction}

Solving the commutation equations for $w_1,w_2$ gives
\[
    C_{M_4(S)}(w_1,w_2)
    =S[Q_4]
    =\{aI_4+bQ_4\mid a,b\in S\}.
\]

\begin{lemma}[Unitary correction]
    \label{lemma:A3-elementary-unitary-correction}
    After passing to a further faithfully flat local extension that preserves the residue field and extends the involution, there exists
    \[
        L=I_4+cQ_4\in\operatorname{GL}_4(S,J_S),
        \qquad c\in J_S\cap S_\theta^-,
    \]
    such that $L$ commutes with $w_1,w_2$ and $[CL]\in G^{(3)}(S)$.
\end{lemma}

\begin{proof}
    Since $M=C^\dagger C$ commutes with $w_1,w_2$, it belongs to $S[Q_4]$.
    The identities $M^\dagger=M$ and $Q_4^\dagger=-Q_4$, together with $M\equiv I_4\pmod{J_S}$, give
    \[
        M=uI_4+vQ_4,
        \qquad
        u\in S_\theta,\quad u\equiv1\pmod{J_S},
        \qquad
        v\in J_S\cap S_\theta^-.
    \]
    For $L=I_4+cQ_4$ with $\bar c=-c$, we have $L^\dagger=L$ and
    \[
        L^\dagger ML
        =\bigl(u(1-c^2)-2vc\bigr)I_4
         +\bigl(2uc+v(1-c^2)\bigr)Q_4.
    \]
    Thus it suffices to find $c$ in the radical, with $\bar c=-c$, satisfying
    \[
        f(c)=0,
        \qquad f(T)\coloneqq vT^2-2uT-v.
    \]

    Since $f(0)=-v\in J_S$ and $f'(0)=-2u\in S^\times$, we consider
    \[
        A\coloneqq S[T]/(f),
        \qquad
        \mathfrak n\coloneqq(J_S,T)A,
        \qquad
        S'\coloneqq A_{\mathfrak n}.
    \]
    We have $A/\mathfrak n\cong k$, so $\mathfrak n$ is maximal and lies over $J_S$.
    The image of $f'(T)$ is a unit in $S'$, and the Jacobian criterion gives flatness over $S$~\cite[Tags~00T7 and~00TA]{StacksProject}.
    Since $S\to S'$ is also local, it is faithfully flat~\cite[Tag~00HR]{StacksProject}, with residue field $k$.

    Extend the involution to $S[T]$ by setting $\overline T=-T$.
    Then $\overline{f(T)}=-f(T)$, so the involution descends to $A$, preserves $\mathfrak n$, and extends to $S'$.
    The image $c$ of $T$ satisfies
    \[
        c\in\operatorname{Rad}(S'),
        \qquad \bar c=-c,
        \qquad f(c)=0.
    \]
    Continue to write $S$ for this extension.
    Then $L=I_4+cQ_4$ is congruent to $I_4$ modulo $J_S$, hence is invertible, and commutes with $w_1,w_2$.
    Moreover,
    \[
        (CL)^\dagger(CL)
        =L^\dagger ML
        =\bigl(u(1-c^2)-2vc\bigr)I_4.
    \]
    The scalar coefficient is congruent to $1$ modulo $J_S$, so it is a unit.
    Hence $[CL]\in G^{(3)}(S)$.
\end{proof}

Replace $C$ by $CL$ and $\psi$ by $i_{[L]}^{-1}\circ\psi$, retaining the same notation.
The Weyl normalization is preserved, and now $[C]\in G^{(3)}(S)$.
By normality of $E^{(3)}(S)$ in $G^{(3)}(S)$~\cite[Corollary~4.3]{SG&DM1}, we have
\[
    \psi(E^{(3)}(R))\subseteq E^{(3)}(S).
\]
For the updated $C$, the matrix $M=C^\dagger C$ is an invertible scalar matrix.
Thus \eqref{eq:weyl-normalized-preserved-form} gives
\[
    B_A^\dagger B_A=I_4.
\]
The matrices $B_A$ are therefore determinant-one unitary representatives of $\psi([A])$, with reduction $\mu_k(A_0)$.
In particular, the representatives of $\psi(u_{\pm i})$ are congruent to $u_{\pm i}$ modulo $J_S$.


\subsubsection{Normalization of the images of \texorpdfstring{$u_i$}{u_i}}

In this subsubsection, we normalize the images of $u_1,u_2$ to one-parameter forms.
We retain the Weyl and unitary normalizations over the auxiliary local ring $S$, so
$$
    \psi=i_{[C]}^{-1}\circ\varphi_1,
    \qquad C\in\operatorname{GL}_4(S,J_S),
    \qquad [C]\in G^{(3)}(S).
$$
For $i=1,2$, choose determinant-one unitary representatives $X_i$ of $\psi(u_i)$ with $X_i\equiv u_i\pmod{J_S}$.

We first establish the descent lemma needed in the proof.

\begin{lemma}[Descent of a conjugating class]
    \label{lem:matrix-algebra-descent}
    Let $T/B$ be a faithfully flat local extension of commutative local rings with $2\in B^\times$.
    Assume that the involution on $T$ extends that on $B$ and that the extension induces an isomorphism of residue fields.
    Write $\mathfrak m_B\coloneqq\operatorname{Rad}(B)$ and $\mathfrak m_T\coloneqq\operatorname{Rad}(T)$.
    Suppose that
    \[
        C\in\operatorname{GL}_4(T,\mathfrak m_T),
        \qquad [C]\in G^{(3)}(T),
    \]
    and that $V_1,\ldots,V_r\in\operatorname{GL}_4(B)$ generate $M_4(B)$ as a $B$-algebra and satisfy
    \[
        CV_iC^{-1}\in\operatorname{GL}_4(B)
        \qquad(i=1,\ldots,r).
    \]
    Then $C=\eta C_B$ for some $\eta\in T^\times$ and
    \[
        C_B\in\operatorname{GL}_4(B,\mathfrak m_B),
        \qquad [C_B]\in G^{(3)}(B).
    \]
\end{lemma}

\begin{proof}
    Faithful flatness allows us to identify $B$ with a subring of $T$, and locality gives
    \[
        B\cap\mathfrak m_T=\mathfrak m_B.
    \]
    Since the $V_i$ generate $M_4(B)$ and their conjugates have entries in $B$, conjugation by $C$ defines a $B$-algebra homomorphism
    \[
        F\colon M_4(B)\longrightarrow M_4(B),
        \qquad F(A)=CAC^{-1}.
    \]
    Put $\alpha\coloneqq(C^{-1})_{11}$.
    Since $C\equiv I_4\pmod{\mathfrak m_T}$, we have $\alpha\in1+\mathfrak m_T$, so $\alpha$ is a unit.
    For each standard basis vector $e_i$,
    \[
        \alpha Ce_i
        =CE_{i1}C^{-1}e_1
        =F(E_{i1})e_1\in B^4.
    \]
    Thus $C_B\coloneqq\alpha C$ has entries in $B$.
    It is congruent to $I_4$ modulo $\mathfrak m_T$, hence also modulo $\mathfrak m_B$.
    Therefore
    \[
        C_B\in\operatorname{GL}_4(B,\mathfrak m_B),
        \qquad C=\alpha^{-1}C_B.
    \]

    Finally, $[C_B]=[C]\in G^{(3)}(T)$, so
    \[
        C_B^\dagger C_B=\lambda I_4
        \qquad\text{for some }\lambda\in T^\times.
    \]
    The left-hand side has entries in $B$ and is congruent to $I_4$ modulo $\mathfrak m_B$.
    Hence $\lambda\in1+\mathfrak m_B\subseteq B^\times$, proving that $[C_B]\in G^{(3)}(B)$.
\end{proof}

Put
\[
    X_3\coloneqq w_2X_1w_2^{-1},\qquad
    X_4\coloneqq w_1X_2w_1^{-1},\qquad
    X_{-2}\coloneqq w_2X_2^{-1}w_2^{-1}.
\]
The projective relations used below hold as exact matrix identities: the representatives have determinant one, and each scalar factor has fourth power $1$ and is congruent to $1$ modulo $J_S$.
Thus Lemma~\ref{lem:scalar-normalization} applies.

\begin{lemma}[Preliminary normal form]
    \label{lemma:A3-elementary-preliminary-normal-form}
    After a further faithfully flat local extension preserving the residue field and extending the involution, and conjugation by
    \[
        L=I_4+\tau Q_4,\qquad \tau\in J_S\cap S_\theta,
    \]
    we may assume
    \begin{equation}\label{eq:A3-elementary-X2-preliminary}
        X_2=
        \begin{pmatrix}
            a&0&0&0\\
            0&b&a^{-1}&0\\
            0&0&e&0\\
            0&0&0&a
        \end{pmatrix},
        \qquad a,b,e\equiv1\pmod{J_S},
    \end{equation}
    where
    \begin{equation}\label{eq:A3-elementary-abe-relations}
        a^3=1,\qquad be=a,\qquad b+e=a+1,\qquad a\bar a=1,
    \end{equation}
    and
    \begin{equation}\label{eq:A3-elementary-unitary-be}
        b\bar e=\bar b e=1.
    \end{equation}
    Moreover,
    \begin{equation}\label{eq:A3-elementary-X1-preliminary}
        X_1=
        \begin{pmatrix}
            p&q&r&s\\
            0&v&0&n\\
            r&s&p&q\\
            0&n&0&v
        \end{pmatrix},
    \end{equation}
    with
    \begin{equation}\label{eq:A3-elementary-pqrs-formulas}
        \begin{aligned}
            q&=\frac{a^2+a}{2},&
            s&=\frac{a-a^2}{2},\\
            p&=\frac{a^2e+a-e+1}{2},&
            v&=\frac{a^2+e+1-a^2e}{2},\\
            r&=\frac{1+a^2}{2}(e-a),&
            n&=\frac{1+a^2}{2}(1-e).
        \end{aligned}
    \end{equation}
    The matrix $L$ centralizes $w_1,w_2$ and is a unitary similitude.
    The conjugation replaces $C$ by $CL^{-1}$ and $\psi$ by $i_{[L]}\circ\psi$.
\end{lemma}

\begin{proof}
    \emph{Triangularizing $X_2$.}
    The relations $[X_2,h_2]=[X_2,w_4]=1$ give
    \[
        X_2=
        \begin{pmatrix}
            a&0&0&z\\
            0&b&c&0\\
            0&d&e&0\\
            -z&0&0&a
        \end{pmatrix},
        \qquad a,b,c,e\equiv1\pmod{J_S},\quad z,d\in J_S.
    \]
    From $[X_2,X_4]=1$ we obtain $z(c+d)=0$; since $c+d$ is a unit, $z=0$.
    For $L=I_4+\tau Q_4$ with $\bar\tau=\tau$,
    \[
        L^\dagger L=(1+\tau^2)I_4,\qquad
        (LX_2L^{-1})_{32}
        =\frac{d+(b-e)\tau-c\tau^2}{1+\tau^2}.
    \]
    We therefore seek a symmetric root in the radical of
    \[
        g(T)\coloneqq T^2-c^{-1}(b-e)T-c^{-1}d.
    \]
    Put $\Delta=be-cd$.
    Unitarity of the middle block gives
    \[
        (\bar b,\bar c,\bar d,\bar e)
        =\Delta^{-1}(b,c,d,e),
    \]
    so both coefficients of $g$ are symmetric.
    The algebra $S'=S[T]/(g)$ is free of rank two over $S$ and has residual algebra $k[T]/(T^2)$.
    Since it is finite over the local ring $S$, it is local with maximal ideal $(J_S,T)S'$ and residue field $k$.
    Thus $S'/S$ is faithfully flat and local, and the involution extends by $\bar T=T$.
    The image $\tau$ of $T$ lies in the radical and gives the required conjugation.
    This monic quadratic construction does not require a simple residual root.

    Retain the notation after conjugation, so that $d=0$.
    The relation $(w_2X_2)^3=I_4$ gives
    \[
        a^3=c^3=1,\qquad be=c^2.
    \]
    Since $\det X_2=a^2be=1$, we have $(ac)^2=1$ and $ac\equiv1\pmod{J_S}$.
    Hence $c=a^{-1}$ by Lemma~\ref{lem:scalar-normalization}, and $be=a$.
    Unitarity gives $a\bar a=b\bar e=\bar b e=1$.

    \smallskip
    \noindent
    \emph{Determining $X_1$.}
    The relations
    \[
        h_2X_1h_2^{-1}=w_3X_1w_3^{-1}=X_1^{-1}
    \]
    imply that $X_1$ commutes with $h_2w_3$, so it has the form
    \[
        X_1=
        \begin{pmatrix}
            p&q&r&s\\
            t&v&j&n\\
            r&s&p&q\\
            j&n&t&v
        \end{pmatrix},
        \qquad p,q,v\equiv1\pmod{J_S}.
    \]
    The linear equations $[X_1,X_4]=[X_1,X_{-2}]=1$ give
    \[
        j=0,\qquad
        r=aq(e-a)=aq(1-b),\qquad
        n=aq(b-a),\qquad t=ar(b-a).
    \]
    Since $aq$ is a unit, $b+e=a+1$.
    Together with $be=a$, this gives
    \[
        (e-1)(e-a)=0,\qquad (e-a)(b-a)=0,
    \]
    and hence $t=rn=0$.
    We obtain~\eqref{eq:A3-elementary-X1-preliminary}, with
    $r=aq(e-a)$ and $n=aq(1-e)$.

    The relation with $h_2$ and the rank-one relation $(X_1w_1)^3=I_4$ give
    \[
        p^2-r^2=1,\qquad v^2-n^2=1,\qquad
        q=pv,\qquad ps+qr+n=0.
    \]
    In particular, $q^2=1+(r+n)^2$ and $r+n=aq(1-a)$, so
    \[
        q^2(3-a-a^2)=1.
    \]
    The element $q_0=(a+a^2)/2$ satisfies the same equation, using $a^3=1$.
    Since $3-a-a^2\equiv1$ and $q+q_0\equiv2$ modulo $J_S$, both are units, and $q=q_0$.
    With this value of $q$, the proposed expressions for $p$ and $v$ in~\eqref{eq:A3-elementary-pqrs-formulas} have squares $1+r^2$ and $1+n^2$, respectively; this follows by reduction using $a^3=1$ and $(e-1)(e-a)=0$.
    Each is congruent to $1$, so the same difference-of-squares argument determines $p,v$ uniquely.
    Finally, $ps+qr+n=0$ determines $s$, and the formulas for $r,n$ follow from those above.
    This proves all six expressions without any additional annihilator condition.
\end{proof}

For the next two lemmas, use the polynomial matrices
\begin{equation}\label{eq:A3-X2-delta-normal-form}
    X_2^\delta\coloneqq
    \begin{pmatrix}
        1+2\delta&0&0&0\\
        0&1+\delta&1-2\delta&0\\
        0&0&1+\delta&0\\
        0&0&0&1+2\delta
    \end{pmatrix},
\end{equation}
\begin{equation}\label{eq:A3-X1-delta-normal-form}
    X_1^\delta\coloneqq
    \begin{pmatrix}
        1&1&-\delta&-\delta\\
        0&1&0&-\delta\\
        -\delta&-\delta&1&1\\
        0&-\delta&0&1
    \end{pmatrix}.
\end{equation}
They are defined over any coefficient ring; no admissibility of $\delta$ is implicit in this notation.

\begin{lemma}[Surjective coefficient recovery]
    \label{lemma:A3-surjective-coefficient-recovery}
    Let $B$ be a quotient of $R$ by a $\theta$-invariant ideal contained in $J$, and let
    \[
        \rho\colon E_{\mathrm{ad}}^{(3)}(R)
        \longrightarrow E_{\mathrm{ad}}^{(3)}(B)
    \]
    be a surjective homomorphism inducing an automorphism $\mu_k$ of the common residue field that commutes with its involution.
    Suppose that $\rho$ fixes $w_1,w_2$ and that
    \[
        \rho(u_i)=[X_i^\varepsilon]\quad(i=1,2),
        \qquad \varepsilon\in\operatorname{Rad}(B)\cap B_\theta^-,
        \qquad \varepsilon^2=3\varepsilon=0.
    \]
    Then there is a surjective unital ring homomorphism $\mu\colon R\to B$ commuting with the involutions and inducing $\mu_k$, such that
    \[
        \varepsilon\in\mathcal E(B,\theta),
        \qquad \rho=\Theta_\varepsilon^{(3)}\circ\mu.
    \]
    In particular,
    \[
        \rho(x_1(\xi))=x_1(\mu(\xi)),\qquad
        \varepsilon\mu(\xi)=0\quad(\xi\in R_\theta^-),
        \qquad \mu(R_\theta^-)=B_\theta^-.
    \]
\end{lemma}

\begin{proof}
    The commutator construction of Subsubsection~\ref{subsubsec:lifting_remark} applies to $\rho$ as well: perfection of $E_{\mathrm{sc}}^{(3)}(R)$ gives determinant-one unitary representatives over $B$ with the prescribed matrix reductions.
    Thus every projective relation below whose reductions agree is exact by Lemma~\ref{lem:scalar-normalization}.
    The images of the other distinguished root elements are given by the fixed Weyl conjugations.

    \smallskip
    \noindent
    \emph{Symmetric and anti-invariant parameters.}
    For $s\in R_\theta$, the commutation relations of $x_2(s)$ with $h_2,w_4,u_2,u_4,u_{-1}$ give a representative
    \[
        Y_s=\operatorname{diag}(b+c\varepsilon,b,b,b+c\varepsilon)+cE_{23},
        \qquad b\equiv1\pmod{\operatorname{Rad}(B)}.
    \]
    Indeed, commutation with $h_2,w_4$ gives the two blocks; the relations with $X_2^\varepsilon,X_4^\varepsilon$ eliminate the lower entries and equalize the middle diagonal entries, and the relation with $X_{-1}^\varepsilon$ gives the outside entry $b+c\varepsilon$.
    Unitarity gives $b\bar b=1$ and $\bar bc=b\bar c$.
    Consequently, $\nu(s)=c/b$ is symmetric and
    \[
        \rho(x_2(s))=[F_2(\nu(s))],
        \qquad F_2(z)\coloneqq h_1(1+2z\varepsilon)x_2(z).
    \]
    The parameter is the projectively invariant ratio $(Y_s)_{23}/(Y_s)_{22}$.
    The root-subgroup law gives additivity of $\nu$, and $\nu(1)=1$.

    For $\xi\in R_\theta^-$, let $Y$ represent $\rho(x_1(\xi))$.
    Commutation with $w_3$ and $X_4^\varepsilon$ gives
    \[
        Y=
        \begin{pmatrix}
            a&b&-b\varepsilon&c\\
            0&a&0&b\varepsilon\\
            b\varepsilon&c&a&-b\\
            0&-b\varepsilon&0&a
        \end{pmatrix},
        \qquad a\equiv1\pmod{\operatorname{Rad}(B)},\quad
        c\in\operatorname{Rad}(B).
    \]
    The $(1,4)$-entry of $YX_1^\varepsilon-X_1^\varepsilon Y$ is $-4b\varepsilon$, so $b\varepsilon=0$.
    The relation $h_2Yh_2^{-1}=Y^{-1}$ then gives $a^2=1$ and $2ac=0$.
    Hence $a=1,c=0$, and unitarity gives $\bar b=-b$.
    Thus
    \[
        \rho(x_1(\xi))=x_1(\eta(\xi)),
        \qquad \eta(\xi)=b\in B_\theta^-,
        \qquad \varepsilon\eta(\xi)=0.
    \]
    Uniqueness of the parameter and the root-subgroup law make $\eta$ additive.
    Both parameter maps have the reductions prescribed by $\mu_k$.

    \smallskip
    \noindent
    \emph{Recovering multiplication.}
    Define $\mu(t)=\nu(t^+)+\eta(t^-)$.
    This map is additive, unital, commutes with the involutions, and induces $\mu_k$.
    For $z\in B$, with $p=z^+$, put
    \[
        F_1(z)\coloneqq
        \begin{pmatrix}
            1&z&-p\varepsilon&-p^2\varepsilon\\
            0&1&0&-p\varepsilon\\
            -p\varepsilon&-p^2\varepsilon&1&\bar z\\
            0&-p\varepsilon&0&1
        \end{pmatrix}.
    \]
    For $s\in R_\theta$, the identity
    \[
        [x_2(s),x_{-3}(1)]=x_{-1}(-s)x_{-4}(s)
    \]
    determines the image of $x_1(s)$: all other factors have known images, and $x_{-1}(-s)=w_1x_1(s)w_1^{-1}$.
    Substitution gives $\rho(x_1(s))=[F_1(\nu(s))]$.
    Decomposing $t=t^++t^-$ and using $\varepsilon\eta(t^-)=0$, we obtain
    \[
        \rho(x_1(t))=[F_1(\mu(t))]\qquad(t\in R).
    \]

    Apply $\rho$ to
    \[
        [x_1(t),x_2(s)]=x_3(-st)x_4(st\bar t),
        \qquad t\in R,\quad s\in R_\theta.
    \]
    Write $b=\mu(t)$, $p=b^+$, $q=\nu(s)$, $c=\mu(st)$, and $d=\nu(st\bar t)$.
    Comparing the $(2,1)$-, $(1,1)$-, and $(1,3)$-entries gives
    \[
        \varepsilon c=\varepsilon pq,\qquad
        \varepsilon d=\varepsilon p^2q,\qquad
        c=qb+\varepsilon(p^3+p)q^2+\varepsilon cd.
    \]
    Substituting the first two equalities into the third, and using $3\varepsilon=0$, gives
    \[
        \mu(st)=\nu(s)\mu(t)-\varepsilon\nu(s)^2p(p^2-1).
    \]
    Taking $s=1$ first shows $\varepsilon p(p^2-1)=0$.
    Hence
    \begin{equation}\label{eq:A3-surjective-symmetric-product}
        \mu(st)=\mu(s)\mu(t),\qquad
        \varepsilon\mu(t)\bigl(\mu(t)^2-1\bigr)=0
        \qquad(s\in R_\theta,\ t\in R),
    \end{equation}
    where the second equality also uses $\varepsilon\mu(t^-)=0$.
    For $\xi,\zeta\in R_\theta^-$, the identity
    \[
        [x_1(\xi),x_3(\zeta)]=x_4(2\xi\zeta)
    \]
    gives first $\varepsilon\nu(2\xi\zeta)=0$ from the diagonal entries, and then
    $\nu(2\xi\zeta)=2\eta(\xi)\eta(\zeta)$ from the $(1,4)$-entry.
    Since $2$ is invertible, this proves multiplicativity on products of anti-invariant elements.
    Together with~\eqref{eq:A3-surjective-symmetric-product} and $R=R_\theta\oplus R_\theta^-$, it proves that $\mu$ is a ring homomorphism.

    \smallskip
    \noindent
    \emph{Admissibility and surjectivity of the coefficient map.}
    Put $B_\mu=\mu(R)$ and $B_0=B_\mu+\varepsilon B_\mu$.
    The root-image formulas and their fixed Weyl conjugates have representatives and inverses over the subring $B_0$.
    Surjectivity of $\rho$ implies $B=B_0$: for each $b\in B$, a preimage of $x_1(b)$ gives $A=\lambda x_1(b)$ with $A,A^{-1}$ over $B_0$, whence
    \[
        \lambda=A_{11}\in B_0,\qquad
        \lambda^{-1}=(A^{-1})_{11}\in B_0,\qquad
        b=\lambda^{-1}A_{12}\in B_0.
    \]
    If $z\in B_\theta^-$, write $z=\mu(a)+\varepsilon\mu(b)$.
    Taking symmetric parts gives $\mu(a^+)+\varepsilon\mu(b^-)=0$.
    Multiplication by $\varepsilon$, together with $\varepsilon\mu(a^-)=0$, yields $\varepsilon z=0$.
    Thus $\varepsilon B_\theta^-=0$.
    If $u=\mu(a)+\varepsilon\mu(b)$ is a unit, then $\mu(a)$ is a unit as well.
    The second identity in~\eqref{eq:A3-surjective-symmetric-product} gives
    \[
        \varepsilon(u^2-1)=\varepsilon\bigl(\mu(a)^2-1\bigr)=0.
    \]
    Lemma~\ref{lemma:exnl-prmr-space} therefore yields $\varepsilon\in\mathcal E(B,\theta)$.

    The automorphism $\Theta_{-\varepsilon}^{(3)}$ is now available by Proposition~\ref{prop:exp_auto}.
    The surjective homomorphism $\Theta_{-\varepsilon}^{(3)}\circ\rho$ acts through $\mu$ on each root subgroup, by the formulas above.
    Repeating the representative-and-inverse argument with $B_\mu$ instead of $B_0$ gives $B=B_\mu$.
    Thus $\mu$ is surjective, and taking anti-invariant parts of preimages gives $\mu(R_\theta^-)=B_\theta^-$.
    The asserted factorization follows on the root generators.
\end{proof}

\begin{lemma}[Reduction to one residual parameter]
    \label{lemma:A3-elementary-delta-normal-form}
    A further conjugation by $L=I_4+\gamma Q_4$, with $\gamma\in J_S\cap S_\theta$ and $L^\dagger L=I_4$, gives
    \[
        X_1=X_1^\delta,\qquad X_2=X_2^\delta,
        \qquad \delta\in J\cap R_\theta^-,
    \]
    where
    \begin{equation}\label{eq:A3-elementary-delta-basic}
        \delta^2=0,\qquad 3\delta=0.
    \end{equation}
    This conjugation preserves the Weyl normalization.
    The total conjugating class also descends to $G^{(3)}(R)$.
\end{lemma}

\begin{proof}
    \emph{Descending the preliminary parameter.}
    Before the auxiliary extensions, let $C_{\mathrm W}\in\operatorname{GL}_4(R,J)$ be the Weyl-normalizing matrix, and choose a representative $\widetilde u_2$ of $\varphi_1(u_2)$ as in~\eqref{eq:compatible-matrix-representatives}.
    Then
    \[
        Y_0=C_{\mathrm W}^{-1}\widetilde u_2C_{\mathrm W}\in\operatorname{SL}_4(R),
        \qquad Y_0\equiv u_2\pmod J,
    \]
    commutes with $h_2$ exactly, by Lemma~\ref{lem:scalar-normalization}.
    Its partial trace ratio $\kappa_3([Y_0])$ is therefore defined and belongs to $R$.
    Every subsequent conjugating matrix lies in $S[Q_4]^\times$ and preserves the eigenspaces of $h_2$.
    Lemma~\ref{lemma:projective-trace-invariants} and the preliminary form give
    \[
        \kappa\coloneqq\kappa_3([Y_0])
        =\kappa_3([X_2])=\frac{2a}{b+e}=\frac{2a}{a+1}.
    \]
    Since $\kappa\equiv1\pmod J$, the element $2-\kappa$ is a unit in $R$, and
    \[
        a=\frac{\kappa}{2-\kappa}\in R.
    \]
    Faithful flatness lets us identify $R\subseteq S$, with $R\cap J_S=J$.

    Put
    \[
        \zeta=\frac{a-1}{2}\in J,\qquad I=\zeta^2R.
    \]
    The relations $a\bar a=1$ and $a^3=1$ give
    \begin{equation}\label{eq:A3-predefect-parameter-identities}
        \bar\zeta=-\frac{\zeta}{a},\qquad
        3a\zeta=-4\zeta^3.
    \end{equation}
    Thus $I$ is $\theta$-invariant, and the parameter
    $\varepsilon=\zeta+I$ in $B=R/I$ satisfies
    $\varepsilon^2=3\varepsilon=0$ and $\bar\varepsilon=-\varepsilon$.
    These properties are not yet asserted over $R$.

    \smallskip
    \noindent
    \emph{Balancing modulo the defect.}
    Set
    \[
        \gamma=\frac{a(b-e)}2,\qquad L=I_4+\gamma Q_4.
    \]
    Since $\bar b=b/a$, $\bar e=e/a$, and $a^3=1$, we have $\bar\gamma=\gamma$.
    Also,
    \[
        (b-e)^2=(a-1)^2=4\zeta^2,\qquad
        \gamma^2=a^2\zeta^2.
    \]
    Therefore $L\equiv I_4\pmod{J_S}$ is invertible, centralizes $w_1,w_2$, and satisfies
    $L^\dagger L=(1+\gamma^2)I_4$.
    Conjugate by $L$, replacing $C$ by $CL^{-1}$ and retaining the notation $X_i,\psi$.
    Substitution into~\eqref{eq:A3-elementary-pqrs-formulas} gives
    \begin{equation}\label{eq:A3-normal-forms-modulo-defect}
        X_i\equiv X_i^\zeta\pmod{IS}\qquad(i=1,2).
    \end{equation}
    More explicitly, modulo $IS$ write
    $b=1+\zeta+(b-e)/2$ and $e=1+\zeta-(b-e)/2$.
    In $LX_iL^{-1}$, the terms involving $b-e$ cancel using
    $\zeta^2=(b-e)^2=3\zeta=0$.
    No relation $\zeta(b-e)=0$ is used.

    \smallskip
    \noindent
    \emph{A surjective map over $B$.}
    Put $T=S/IS$.
    The extension $T/B$ is faithfully flat and local, with the same residue field and the induced involutions.
    Let $C_T$ denote the reduction of the current $C$, and put
    \[
        V_i=X_i^\varepsilon,\qquad
        V_{-i}=w_iV_i^{-1}w_i^{-1}\qquad(i=1,2).
    \]
    These matrices lie in $\operatorname{SL}_4(B)$ and reduce to $u_{\pm1},u_{\pm2}$.
    Their $B$-algebra is $M_4(B)$: its reduction is $M_4(k)$ by Lemma~\ref{lemma:A3-generates-matrix-algebra}, and Nakayama's lemma applies to the finitely generated quotient module.
    For each $V_{\pm i}$, the class $[C_TV_{\pm i}C_T^{-1}]$ is the reduction of $\varphi_1(u_{\pm i})$.
    Choose the corresponding determinant-one representative over $B$ by reducing~\eqref{eq:compatible-matrix-representatives}.
    The two representatives differ by a scalar congruent to $1$ whose fourth power is $1$; hence they are equal.
    Lemma~\ref{lem:matrix-algebra-descent} therefore gives
    \[
        C_T=\eta C_B,\qquad
        C_B\in\operatorname{GL}_4(B,\operatorname{Rad}(B)),
        \qquad [C_B]\in G^{(3)}(B).
    \]
    Consequently,
    \[
        \rho=i_{[C_B]}^{-1}\circ\lambda_I\circ\varphi_1
        \colon E^{(3)}(R)\longrightarrow E^{(3)}(B)
    \]
    is surjective, using elementary reduction and normality of $E^{(3)}(B)$ in $G^{(3)}(B)$.
    This is a composition with reduction; no invariance of the level-$I$ congruence subgroup under $\varphi_1$ is assumed.
    The map fixes $w_1,w_2$, has the prescribed residual field automorphism, and satisfies $\rho(u_i)=[X_i^\varepsilon]$.

    By Lemma~\ref{lemma:A3-surjective-coefficient-recovery}, there exists $\xi\in R_\theta^-$ such that
    \[
        \rho(x_1(\xi))=x_1(\varepsilon).
    \]
    Choose a determinant-one unitary representative $Y$ of $\psi(x_1(\xi))$ with the prescribed reduction.
    Its reduction over $T$ equals $x_1(\varepsilon)$ as a matrix: the possible scalar is again eliminated by its fourth power and its residue.
    Hence
    \[
        Y=I_4+\zeta N+D,\qquad
        N=E_{12}-E_{34},\qquad D\in M_4(IS).
    \]

    \smallskip
    \noindent
    \emph{Eliminating the defect.}
    Work in $\mathcal M=IS/J_SIS$, and write $d=[\zeta^2]$ for its generator and $D_{ij}$ for the entry classes of $D$.
    Since $\xi$ is anti-invariant, the three commutations
    \[
        Yw_3=w_3Y,\qquad YX_1=X_1Y,\qquad YX_4=X_4Y
    \]
    are exact.
    In the last two equations, replacing $X_i$ by the polynomial candidates changes the residuals only by matrices over $J_SIS$, because $Y-I_4\in M_4(J_S)$.
    The candidates have linear parts
    \[
    \begin{aligned}
        X_1^\zeta&=u_1-\zeta
            (E_{13}+E_{14}+E_{24}+E_{31}+E_{32}+E_{42}),\\
        X_4^\zeta&\coloneqq w_1X_2^\zeta w_1^{-1}
            =u_4+\zeta\bigl(\operatorname{diag}(1,2,2,1)+2E_{14}\bigr).
    \end{aligned}
    \]
    The $(1,1)$- and $(2,2)$-entries of the first commutation, the $(1,2)$- and $(3,4)$-entries of the third, and the $(1,4)$-entry of the second give, respectively,
    \[
    \begin{gathered}
        D_{13}+D_{31}=0,\qquad D_{24}+D_{42}=0,\\
        D_{42}+d=0,\qquad -D_{31}+d=0,\\
        D_{13}-D_{24}-2d=0.
    \end{gathered}
    \]
    Thus $D_{13}=-d$ and $D_{24}=d$, and the last equation gives $-4d=0$.
    Since $2$ is invertible, $d=0$, so $\mathcal M=0$.
    Nakayama's lemma applied to the principal ideal $IS$ gives $IS=0$.
    Faithful flatness then gives $I=0$, proving $(a-1)^2=0$ over $R$.

    Now~\eqref{eq:A3-normal-forms-modulo-defect} is an exact equality.
    Set $\delta=\zeta$.
    Equation~\eqref{eq:A3-predefect-parameter-identities} gives $3\delta=0$ and $\bar\delta=-\delta$, and $\gamma^2=a^2\delta^2=0$ gives $L^\dagger L=I_4$.
    Finally, since $I=0$, the earlier descent over $T/B$ is now descent over $S/R$ itself.
    Thus $C=\eta C_R$ with $C_R\in\operatorname{GL}_4(R,J)$ and $[C_R]\in G^{(3)}(R)$.
\end{proof}

Define
\[
    X_3^\delta\coloneqq w_2X_1^\delta w_2^{-1},
    \qquad
    X_4^\delta\coloneqq w_1X_2^\delta w_1^{-1}.
\]
Then
\[
    \psi(u_i)=[X_i^\delta]\qquad(1\leq i\leq4).
\]
For $\delta=0$, we have $X_i^0=u_i$, so $\psi$ fixes all four distinguished positive root elements.


\subsubsection{Normalization over the original ring}

Lemma~\ref{lemma:A3-elementary-delta-normal-form} shows that both the parameter and the total conjugating class are defined over $R$.
More precisely,
\[
    \delta\in J\cap R_\theta^-,
    \qquad \delta^2=3\delta=0,
\]
and
\[
    C=\eta C_0,
    \qquad
    \eta\in S^\times,
    \qquad
    C_0\in\operatorname{GL}_4(R,J),
    \qquad
    [C_0]\in G^{(3)}(R).
\]
Replacing $C$ by $C_0$ leaves the conjugation unchanged.
By normality of $E^{(3)}(R)$ in $G^{(3)}(R)$~\cite[Corollary~4.3]{SG&DM1}, the normalized map is therefore an automorphism over $R$:
\[
    \psi\coloneqq i_{[C_0]}^{-1}\circ\varphi_1
    \in\operatorname{Aut}(E^{(3)}(R)).
\]
It still induces $\mu_k$ on $E^{(3)}(k)$, since $C_0\equiv I_4\pmod J$, and satisfies
\[
    \psi(u_i)=[X_i^\delta],
    \qquad
    \psi(w_i)=w_i,
    \qquad
    \psi(h_i)=h_i
    \qquad(1\leq i\leq4).
\]

This completes the normalization in type ${}^2A_3$.
The hypotheses of Lemma~\ref{lemma:A3-surjective-coefficient-recovery} now hold with $B=R$ and $\rho=\psi$.
We will apply this lemma in Subsection~\ref{subsec:parameter-maps-first-annihilators}, where the injectivity of $\psi$ will also show that the resulting coefficient map is a ring automorphism.


\subsection{Normalization of the images of \texorpdfstring{$u_i$}{u-i} in type \texorpdfstring{${}^2A_4$}{2A4}}
\label{subsec:normalization_of_u_i_A_4}

In this case, all changes of basis are made over $R$.
We first normalize the root images and then show that the total conjugating matrix is a unitary similitude.


\subsubsection{The preserved form}

Define
\[
    A^\dagger\coloneqq Q_5^{-1}\overline A^{\,t}Q_5
    \qquad(A\in M_5(R)).
\]
For $A\in\operatorname{GL}_5(R)$, we have $[A]\in G^{(4)}(R)$ if and only if $A^\dagger A$ is scalar.

Let $C\in\operatorname{GL}_5(R,J)$ be the Weyl-normalizing matrix from
Proposition~\ref{prop:weyl-normalization-A4}, and put
\[
    \psi\coloneqq i_{[C]}^{-1}\circ\varphi_1,
    \qquad M\coloneqq C^\dagger C.
\]
Then $M^\dagger=M$ and $M\equiv I_5\pmod J$.
For each $A\in E_{\mathrm{sc}}^{(4)}(R)$, choose $\widetilde A$ as in
\eqref{eq:compatible-matrix-representatives} and set
\[
    B_A\coloneqq C^{-1}\widetilde A C.
\]
The matrix $B_A$ represents $\psi([A])$ and satisfies
\begin{equation}\label{eq:A4-weyl-normalized-preserved-form}
    \det B_A=1,
    \qquad
    B_A^\dagger M B_A=M,
    \qquad
    (B_A)_0=\mu_k(A_0).
\end{equation}
Indeed, conjugation preserves the determinant, and $C\equiv I_5\pmod J$ preserves the reduction.
Since $\widetilde A^\dagger\widetilde A=I_5$, we also have
\[
    B_A^\dagger M B_A
    =C^\dagger\widetilde A^\dagger\widetilde A C
    =C^\dagger C=M.
\]

We now show that $M$ commutes with $w_1,w_4$.
For $w\in\{w_1,w_4\}$, we have $w^\dagger=w^{-1}$ and
\[
    [CwC^{-1}]=\varphi_1([w])\in G^{(4)}(R).
\]
Thus $(CwC^{-1})^\dagger(CwC^{-1})$ is scalar, which gives
\[
    w^{-1}Mw=\eta_wM
    \qquad\text{for some }\eta_w\in R^\times.
\]
Reducing modulo $J$ gives $\eta_w\equiv1\pmod J$.
Since $w^4=I_5$, applying this identity four times gives $\eta_w^4=1$.
Lemma~\ref{lem:scalar-normalization} then yields $\eta_w=1$.
Hence $M$ commutes with $w_1,w_4$ and with their common square
\[
    H=w_1^2=w_4^2.
\]


\subsubsection{Changes of basis preserving the Weyl normalization}
\label{subsubsec:A4-weyl-centralizer}

Solving the commutation equations for $w_1,w_4$ gives the commutative algebra
\[
    C_{M_5(R)}(w_1,w_4)
    =\{Z(r,s,u,t,q)\mid r,s,u,t,q\in R\},
\]
where
\[
    Z(r,s,u,t,q)\coloneqq
    \begin{pmatrix}
        r&-s&0&-u&t\\
        s&r&0&-t&-u\\
        0&0&q&0&0\\
        u&-t&0&r&-s\\
        t&u&0&s&r
    \end{pmatrix}.
\]
Moreover,
\[
    Z(r,s,u,t,q)^\dagger
    =Z(\bar r,-\bar s,-\bar u,\bar t,\bar q).
\]
In particular, the matrix $M = C^{\dagger} C$ belongs to this algebra and satisfies $M_{33}\in1+J$.

To normalize the root images, we use matrices
\[
    L\in C_{\operatorname{GL}_5(R)}(w_1,w_4)
    \cap\operatorname{GL}_5(R,J).
\]
Such a matrix need not be unitary.
Under the replacements
\[
    C\longmapsto CL^{-1},
    \qquad
    B_A\longmapsto LB_AL^{-1},
    \qquad
    M\longmapsto(L^{-1})^\dagger ML^{-1},
\]
the Weyl normalization and all identities in
\eqref{eq:A4-weyl-normalized-preserved-form} remain valid.
The updated $M$ also remains self-adjoint, congruent to $I_5$ modulo $J$, and in the common centralizer.

We continue to regard $\psi$ as a homomorphism into $\operatorname{PGL}_5(R)$.
Once the root images are normalized, Lemma~\ref{lemma:A4-elementary-delta-normal-form} will show that $C^\dagger C$ is scalar, so that $[C]\in G^{(4)}(R)$.


\subsubsection{Normalization of the images of \texorpdfstring{$u_i$}{u_i}}

We normalize the images of $u_1,u_2$ to one-parameter forms. 
Unlike in the case ${}^2 A_3$, here all changes of basis are made over $R$ within the common centralizer of $w_1,w_4$.
We first normalize the third column of the image of $u_2$ by solving a linear system.
A coefficient-recovery lemma over a quotient of $R$ then supplies an additional root element whose image forces the remaining defect to vanish.
This lemma is proved assuming the candidate normal forms over the quotient, independently of their existence over $R$.
The unitary similitude condition on the total conjugating matrix is established at the end.

Let $C_0\in\operatorname{GL}_5(R,J)$ denote the Weyl-normalizing matrix from
Proposition~\ref{prop:weyl-normalization-A4}, and put
\[
    \psi_0\coloneqq i_{[C_0]}^{-1}\circ\varphi_1,
    \qquad M_0\coloneqq C_0^\dagger C_0.
\]
For $i=1,2$, choose determinant-one unitary representatives $\widetilde u_i$ of $\varphi_1(u_i)$ as in \eqref{eq:compatible-matrix-representatives}, and set
\[
    X_i\coloneqq C_0^{-1}\widetilde u_iC_0.
\]
These matrices represent $\psi_0(u_i)$ and satisfy
\[
    \det X_i=1,
    \qquad
    X_i\equiv u_i\pmod J,
    \qquad
    X_i^\dagger M_0X_i=M_0
    \qquad(i=1,2).
\]
By the preserved-form argument above, $M_0$ commutes with $w_1,w_4$ and hence with $H$.

Set $X_3\coloneqq w_1X_2w_1^{-1}$ and
\begin{equation}\label{eq:A4-elementary-V2-word}
    V_2\coloneqq X_2\bigl(w_4w_1X_2w_1^{-1}w_4^{-1}\bigr)^2X_2,
    \qquad V_3\coloneqq w_1V_2w_1^{-1}.
\end{equation}
We first rescale $X_1,X_2$ so that two projective relations hold as exact matrix identities.
Write
\[
    (X_1w_1)^3=\alpha I_5,
    \qquad V_2^2=\beta I_5.
\]
The determinant, preserved-form, and reduction identities imply
\[
    \alpha^5=\beta^5=1,
    \qquad
    \alpha\bar\alpha=\beta\bar\beta=1,
    \qquad
    \alpha,\beta\equiv1\pmod J.
\]
Replace $X_1$ by $\alpha^3X_1$ and $X_2$ by $\beta^3X_2$, and redefine $X_3,V_2,V_3$ using the same formulas.
These rescalings preserve the projective classes, determinants, form identities, and reductions.
The first replacement gives
\[
    \bigl(\alpha^3X_1w_1\bigr)^3=\alpha^{10}I_5=I_5.
\]
The second multiplies $V_2$ by $\beta^{12}=\beta^2$, so its square becomes $\beta^5I_5=I_5$.
Retaining the same notation, we obtain
\begin{equation}\label{eq:A4-exact-finite-relations}
    (X_1w_1)^3=I_5,
    \qquad V_2^2=I_5.
\end{equation}

Next, $X_1$ commutes with $H$ projectively.
Since $H^2=I_5$, the scalar factor has square $1$; it is also congruent to $1$ modulo $J$.
Lemma~\ref{lem:scalar-normalization} therefore makes this commutation exact.
Thus $X_1$ preserves the $+1$ eigenspace $Re_3$ of $H$, and we set
\[
    \delta\coloneqq1-(X_1)_{33}\in J.
\]
The matrix $M_0$ also preserves this line, with $(M_0)_{33}\in R^\times$.
The $(3,3)$-entry of $X_1^\dagger M_0X_1=M_0$ therefore gives
\begin{equation}\label{eq:A4-bar-delta-exact}
    (1-\bar\delta)(1-\delta)=1,
    \qquad
    \bar\delta=-\frac{\delta}{1-\delta}.
\end{equation}
Every matrix commuting with $w_1,w_4$ also commutes with $H$ and hence preserves its two eigenspaces.
Its restriction to $Re_3$ is scalar, so conjugation by such a matrix leaves $(X_1)_{33}$, and therefore $\delta$, unchanged.

\begin{lemma}[Linear normalization of the third column]
    \label{lemma:A4-elementary-residual-column-normalization}
    There exists
    \[
        L=Z(1,s,u,t,1+\omega)\in\operatorname{GL}_5(R,J)
    \]
    centralizing $w_1,w_4$ such that, after replacing $X_i$ by $LX_iL^{-1}$,
    \[
        (X_2)_{23}=1,\qquad (X_2)_{43}=0,\qquad
        (X_2)_{13}=(X_2)_{53}=-\delta.
    \]
    No symmetry conditions on $s,u,t,\omega$ and no extension of $R$ are needed.
\end{lemma}

\begin{proof}
    Write $X_2e_3=(a,b,c,d,f)^t$, so $b\equiv c\equiv1\pmod J$ and $a,d,f\in J$.
    Since $Le_3=(1+\omega)e_3$, the required conditions are equivalent to the linear system
    \[
    \begin{aligned}
        sa+b-td-uf&=1+\omega,\\
        ua-tb+d-sf&=0,\\
        a-sb-ud+tf&=-\delta(1+\omega),\\
        ta+ub+sd+f&=-\delta(1+\omega).
    \end{aligned}
    \]
    In the variable order $(\omega,t,s,u)$, its coefficient matrix reduces to $\operatorname{diag}(-1,-1,-1,1)$ modulo $J$.
    It is therefore invertible over $R$, and the unique solution has all four entries in $J$.
    Hence $L\equiv I_5\pmod J$ is invertible and belongs to the required centralizer.
\end{proof}

Put $C\coloneqq C_0L^{-1}$ and
$\psi\coloneqq i_{[C]}^{-1}\circ\varphi_1$, retaining the notation for the conjugated matrices.
The identities in \eqref{eq:A4-exact-finite-relations} remain valid.
Define
\[
    M\coloneqq\bigl((C^\dagger C)_{33}\bigr)^{-1}C^\dagger C.
\]
Then
\[
    X_i^\dagger M X_i=M\qquad(i=1,2),
    \qquad M^\dagger=M,
    \qquad M_{33}=1.
\]
Moreover, $M$ commutes with $w_1,w_4$ and is congruent to $I_5$ modulo $J$.
Equation~\eqref{eq:A4-bar-delta-exact} also remains valid, since the conjugation leaves $\delta$ unchanged.
At this stage, $C$ is not yet known to be a unitary similitude.

We use the following polynomial matrices, defined for every $\delta\in J$:
\begin{equation}\label{eq:A4-X1-delta-normal-form}
    X_1^\delta=
    \begin{pmatrix}
        1+\delta&1+\delta&0&0&0\\
        0&1+\delta&0&0&0\\
        0&0&1-\delta&0&0\\
        0&0&0&1+\delta&1+\delta\\
        0&0&0&0&1+\delta
    \end{pmatrix},
\end{equation}
\begin{equation}\label{eq:A4-X2-delta-normal-form}
    X_2^\delta=
    \begin{pmatrix}
        1&0&-\delta&-2\delta&0\\
        \delta&1&1&1/2&\delta\\
        -\delta&0&1&1&-\delta\\
        0&0&0&1&0\\
        0&0&-\delta&-2\delta&1
    \end{pmatrix}.
\end{equation}
At this stage these are only candidate normal forms.

\begin{lemma}[Surjective coefficient recovery]
    \label{lemma:A4-surjective-coefficient-recovery}
    Let $B$ be a quotient of $R$ by a $\theta$-invariant ideal contained in $J$, and let
    \[
        \rho\colon E_{\mathrm{ad}}^{(4)}(R)\longrightarrow E_{\mathrm{ad}}^{(4)}(B)
    \]
    be a surjective homomorphism inducing an automorphism of the common residue field that commutes with its involution.
    Suppose that $\rho$ fixes $w_1,w_4,v_2$ and that
    \[
        \rho(u_i)=[X_i^\varepsilon]\quad(i=1,2),
        \qquad \varepsilon\in\operatorname{Rad}(B)\cap B_\theta^-,
        \quad \varepsilon^2=3\varepsilon=0.
    \]
    Then there is a surjective unital ring homomorphism $\mu\colon R\to B$ commuting with the involutions, such that
    \[
        \rho(x_1(\xi))=x_1(\mu(\xi)),\qquad
        \varepsilon\mu(\xi)=0
        \qquad(\xi\in R_\theta^-).
    \]
    In particular, $\mu(R_\theta^-)=B_\theta^-$.
    Moreover, $\varepsilon\in\mathcal E(B,\theta)$ and $\rho=\Theta_\varepsilon^{(4)}\circ\mu$.
\end{lemma}

\begin{proof}
    We prove the coefficient-recovery statement independently of the normalization argument.
    Set
    \[
        X_3^\varepsilon\coloneqq w_1X_2^\varepsilon w_1^{-1},
        \qquad
        X_4^\varepsilon\coloneqq v_2(X_1^\varepsilon)^2v_2^{-1}.
    \]
    These matrices represent $\rho(u_3)$ and $\rho(u_4)$, respectively.
    All relations below are initially understood as equalities of projective classes; representatives may be rescaled as needed.
    For $s\in R_\theta$, the commutation relations of $x_1(s)$ with $H,w_4,u_1,u_4,u_3$ give a representative
    \[
        \begin{pmatrix}
            a&b&0&0&0\\
            0&a&0&0&0\\
            0&0&a+b\varepsilon&0&0\\
            0&0&0&a&b\\
            0&0&0&0&a
        \end{pmatrix},\qquad a\in B^\times.
    \]
    Indeed, commutation with $H,w_4$ gives the two repeated $2\times2$ blocks, commutation with $X_1^\varepsilon,X_4^\varepsilon$ makes their diagonal entries equal and eliminates the lower entries, and commutation with $X_3^\varepsilon$ gives the middle entry $a+b\varepsilon$.
    The projective scalars are removed successively using $H^2=1$, $w_4^2=H$, and the actions on the middle line; the final scalar is determined by the $(1,1)$-entry.
    The form identity implies that $\nu(s)=b/a$ is symmetric.
    Consequently,
    \[
        \rho(x_1(s))=D(\varepsilon\nu(s))x_1(\nu(s)),
        \qquad \nu(1)=1.
    \]
    Here $D(c)=\operatorname{diag}(1+c,1+c,1-c,1+c,1+c)$ for the square-zero parameters being used.
    The parameter $b/a$ is projectively well-defined, and the root-subgroup law makes $\nu$ additive.

    For $\xi\in R_\theta^-$, normalize the $(4,4)$-entry of a representative $Y$ of $\rho(x_1(\xi))$ to $1$.
    The relations
    \[
        [Y,H]=[Y,X_1^\varepsilon]=[Y,X_3^\varepsilon]=1,
        \qquad w_4Yw_4^{-1}=Y^{-1}
    \]
    give
    \[
        Y=
        \begin{pmatrix}
        1-2b\varepsilon&b&0&-c\varepsilon&c\\
        0&1-2b\varepsilon&0&0&-c\varepsilon\\
        0&0&1-b\varepsilon&0&0\\
        -c\varepsilon&c&0&1&-b+(b^2+c^2)\varepsilon\\
        0&-c\varepsilon&0&0&1
        \end{pmatrix},\qquad c\in\operatorname{Rad}(B).
    \]
    Apply $\rho$ to the elementary identity
    \[
        [x_1(\xi),u_4]=[u_1,v_2x_1(\xi)^{-2}v_2^{-1}].
    \]
    Both commutators have $(3,3)$-entry $1$, so the projective scalar is $1$.
    Comparison of their $(1,5)$- and $(1,2)$-entries gives
    \[
        c(1/2-2b\varepsilon)=0,
        \qquad -c+bc\varepsilon=-(b+c)\varepsilon.
    \]
    The first coefficient is a unit, so $c=0$ and then $\varepsilon b=0$.
    The form identity now gives $\bar b=-b$.
    Thus $\rho(x_1(\xi))=x_1(\eta(\xi))$ with $\eta(\xi)=b\in B_\theta^-$ and $\varepsilon\eta(\xi)=0$.
    Uniqueness of the parameter and the root-subgroup law make $\eta$ additive.

    Define
    \[
        \mu(a)=\nu(a^+)+\eta(a^-),\qquad a\in R.
    \]
    This map is additive, unital, commutes with the involutions, and induces the given residue-field automorphism.
    The preceding formulas give
    \[
        \rho(x_1(a))=D(\varepsilon\mu(a))x_1(\mu(a)).
    \]
    To verify multiplicativity, put $q(t)=x_2(t,t\bar t/2)$ and use
    \[
        q(t)=w_1^{-1}\bigl([x_1(t),u_2]x_4(t/2)\bigr)w_1.
    \]
    If $b=\mu(t)$ and $p=(b+\bar b)/2$, substitution gives
    \[
        \rho(q(t))=
        \left[
        \begin{pmatrix}
            1&0&-p\varepsilon&p^2\varepsilon&0\\
            p^2\varepsilon&1&b&b\bar b/2+p(p^2-1)\varepsilon&p^2\varepsilon\\
            -p\varepsilon&0&1&\bar b&-p\varepsilon\\
            0&0&0&1&0\\
            0&0&-p\varepsilon&p^2\varepsilon&1
        \end{pmatrix}
        \right].
    \]
    In the identity
    \[
        [x_1(a),q(t)]
        =w_1q(at)w_1^{-1}x_4(-at\bar t/2),
    \]
    the first short-root parameter is the ratio of the $(1,3)$- and $(3,3)$-entries.
    Comparing this ratio after applying $\rho$, and using $3\varepsilon=0$ and $\varepsilon\mu(R_\theta^-)=0$, gives
    \[
        \mu(at)=\mu(a)\mu(t)
          +\varepsilon\mu(a)^2\mu(t)\bigl(\mu(t)^2-1\bigr).
    \]
    Taking $a=1$ first yields
    \[
        \varepsilon\mu(t)\bigl(\mu(t)^2-1\bigr)=0\qquad(t\in R).
    \]
    Substitution back gives $\mu(at)=\mu(a)\mu(t)$.
    This also eliminates the extra cubic term in the displayed image of $q(t)$.
    Finally,
    \[
        x_2(0,z)=[x_{-1}(z/2),u_4]\quad(z\in R_\theta^-),
        \qquad
        x_2(t,u)=q(t)x_2(0,u-t\bar t/2)
    \]
    determine the arbitrary short-root images.
    They are exactly the polynomial root prescriptions of Subsection~\ref{subsec:expl_elt}, with parameters mapped by $\mu$ and correction parameter $\varepsilon$.
    This statement concerns the formulas only; admissibility of $\varepsilon$ is established next.

    Put $B_\mu=\mu(R)$ and $B_0=B_\mu+\varepsilon B_\mu$.
    All these root images and their inverses have entries in the subring $B_0$.
    Surjectivity of $\rho$ implies $B=B_0$: for any $b\in B$, a preimage of $x_1(b)$ gives a matrix $A=\lambda x_1(b)$ with $A,A^{-1}$ over $B_0$, and hence
    \[
        \lambda=A_{11}\in B_0,\qquad
        \lambda^{-1}=(A^{-1})_{11}\in B_0,\qquad
        b=\lambda^{-1}A_{12}\in B_0.
    \]
    For $\zeta\in B_\theta^-$, write $\zeta=\mu(a)+\varepsilon\mu(b)$.
    Taking symmetric parts gives $\mu(a^+)+\varepsilon\mu(b^-)=0$, so $\varepsilon\zeta=0$.
    Thus $\varepsilon B_\theta^-=0$.
    If $u=\mu(a)+\varepsilon\mu(b)$ is a unit, then $\mu(a)$ is a unit as well, and the polynomial identity just proved gives
    \[
        \varepsilon(u^2-1)
        =\varepsilon\bigl(\mu(a)^2-1\bigr)=0.
    \]
    Lemma~\ref{lemma:exnl-prmr-space} now yields $\varepsilon\in\mathcal E(B,\theta)$.

    The homomorphism $\Theta_{-\varepsilon}^{(4)}\circ\rho$ is surjective and acts through $\mu$ on every root subgroup.
    Repeating the preceding representative-and-inverse argument with $B_\mu$ in place of $B_0$ gives $B=B_\mu$.
    Therefore $\mu$ is surjective.
    Taking anti-invariant parts of preimages gives $\mu(R_\theta^-)=B_\theta^-$, and the asserted factorization follows on the root generators.
\end{proof}

Returning to the normalized map $\psi$, set
\[
    X_4 \coloneqq V_2X_1^2V_2^{-1}.
\]
This matrix represents $\psi(u_4)$.

\begin{lemma}[One-parameter residual normal form]
    \label{lemma:A4-elementary-delta-normal-form}
    After the linear column normalization, one has $X_i=X_i^\delta$ for $i=1,2$, where
    \begin{equation}\label{eq:A4-delta-relations}
        \delta\in J\cap R_\theta^-,\qquad
        \delta^2=0,\qquad 3\delta=0.
    \end{equation}
    Moreover, $[C]\in G^{(4)}(R)$, $V_2=v_2$, $V_3=v_3$, and
    \begin{equation}\label{eq:A4-X3-delta-normal-form}
        X_3=X_3^\delta\coloneqq w_1X_2^\delta w_1^{-1},
    \end{equation}
    \begin{equation}\label{eq:A4-X4-delta-normal-form}
        X_4=X_4^\delta\coloneqq v_2(X_1^\delta)^2v_2^{-1}.
    \end{equation}
\end{lemma}

\begin{proof}
    Write $X_i=(I_5+A_i)X_i^\delta$ with $A_i\in M_5(J)$, and put $e=\delta^2$, $q=3\delta$.
    Let $I$ be the $\theta$-invariant ideal generated by the entries of $A_1,A_2$ and by $e,q$.
    It is finitely generated and contained in $J$.
    In the $k$-vector space $\mathcal M=I/JI$, products of two defects and products of a defect with $\delta$ vanish.
    Equation~\eqref{eq:A4-bar-delta-exact} therefore gives
    \begin{equation}\label{eq:A4-defect-involution}
        \bar e=e,\qquad\bar q=-q\qquad\text{in }\mathcal M,
    \end{equation}
    while $\delta q=3e$ gives
    \begin{equation}\label{eq:A4-intrinsic-3e}
        3e=0\qquad\text{in }\mathcal M.
    \end{equation}

    \smallskip
    \noindent
    \emph{The algebraic defect equations.}
    Using the matrices defined in \eqref{eq:A4-elementary-V2-word}, put
    \[
        V_2'\coloneqq w_4V_2w_4^{-1},
        \qquad
        T\coloneqq V_2X_1V_2^{-1}.
    \]
    Then $X_4=T^2$.
    We use~\eqref{eq:A4-exact-finite-relations}, $\det X_1=\det X_2=1$, and the projective relations
    \begin{align*}
        [HX_1H^{-1}]&=[X_1],&
        [w_4X_1w_4^{-1}]&=[X_1],\\
        [HX_2H^{-1}]&=[X_2^{-1}],&
        [HV_2H^{-1}]&=[V_2],\\
        \pi([V_2,V_3])&=1,&
        \pi([V_2,V_2'])&=1,\\
        [V_2X_3V_2^{-1}]&=[X_3^{-1}],&
        [V_3X_2V_3^{-1}]&=[X_2^{-1}],\\
        [V_2'X_2(V_2')^{-1}]&=[X_2^{-1}],&
        \pi([X_2,X_3])&=[X_4^{-1}],\\
        \pi([X_1,T])&=1,&
        \pi([T,X_2])&=1,\\
        \pi([T,X_3])&=1,&
        \pi([X_1,X_3])&=1.
    \end{align*}
    Here $[A]$ and $\pi(A)$ denote the projective class, while $[A,B]$ denotes the commutator.
    For a projective equality $[F]=[G]$ with $G_{11}$ a unit, use
    \begin{equation}\label{eq:A4-scalar-free-projective-equality}
        \mathcal P(F,G)_{ij}\coloneqq F_{ij}G_{11}-G_{ij}F_{11}=0.
    \end{equation}
    For $[F]=1$, use
    \begin{equation}\label{eq:A4-scalar-free-projective-identity}
        F_{ij}=0\quad(i\ne j),\qquad
        F_{ii}-F_{11}=0\quad(i=2,\ldots,5).
    \end{equation}
    Include also the normalized coordinates
    \begin{equation}\label{eq:A4-fixed-residual-coordinates}
        \begin{gathered}
        (X_2-X_2^\delta)_{23}=(X_2-X_2^\delta)_{43}
        =(X_2-X_2^\delta)_{13}=(X_2-X_2^\delta)_{53}=0,\\
        (X_1-X_1^\delta)_{33}=0.
        \end{gathered}
    \end{equation}
    Passing these equations to $\mathcal M$ gives a linear system with a selected $50\times50$ coefficient matrix of determinant $-2^{20}$.
    This certificate is verified in the ancillary file
    \[
        \texttt{verify\_A4\_normalization\_repair.py}.
    \]
    No unitary identity for $X_i$ is used in this subsystem.
    Consequently, the column $\mathbf a$ of the fifty entry classes satisfies
    \begin{equation}\label{eq:A4-defect-solution-span}
        \mathbf a=P_e e+P_q q,
        \qquad P_e,P_q\in\mathbb Z[1/2]^{50}.
    \end{equation}
    By~\eqref{eq:A4-defect-involution}, $e,q$ also generate the conjugate entry classes, and hence all of $\mathcal M$.
    The $(3,3)$-entry of the first relation in~\eqref{eq:A4-exact-finite-relations} gives
    \begin{equation}\label{eq:A4-final-rank-defect}
        3e-q=0\qquad\text{in }\mathcal M.
    \end{equation}
    Thus $q=0$ in $\mathcal M$.
    It remains to eliminate $e$ without assuming that the linear column normalization was unitary.

    \smallskip
    \noindent
    \emph{Using surjectivity.}
    Put $B=R/I$ and $\varepsilon=\delta+I$.
    Modulo $I$, the matrices $X_i$ equal $X_i^\varepsilon$, with $\varepsilon^2=3\varepsilon=0$ and $\bar\varepsilon=-\varepsilon$.
    These matrices are unitary, and their defining Weyl word gives $V_2=v_2$ modulo $I$.
    Since they preserve $M$, its reduction commutes with $w_1,w_4,v_2,X_1^\varepsilon,X_2^\varepsilon$.
    These matrices generate $M_5(B)$: their reductions generate $M_5(k)$ by Lemma~\ref{lemma:A4-generates-matrix-algebra}, and Nakayama's lemma applies to the quotient by the generated $B$-algebra.
    Therefore $M$ is scalar modulo $I$; as $M_{33}=1$, it is $I_5$ modulo $I$.

    Let $C_B\in\operatorname{GL}_5(B,\operatorname{Rad}(B))$ denote the reduction of $C$ modulo $I$.
    Since $M\equiv I_5\pmod I$, the matrix $C_B^\dagger C_B$ is scalar.
    Hence $[C_B]\in G^{(4)}(B)$.

    It follows that reduction of $\psi$ gives a surjective homomorphism
    \[
        \rho=i_{[C_B]}^{-1}\circ\lambda_I\circ\varphi_1
        \colon E^{(4)}(R)\longrightarrow E^{(4)}(B).
    \]
    Here we use surjectivity of $\varphi_1$ and of elementary reduction, and normality of $E^{(4)}(B)$ in $G^{(4)}(B)$.
    We do not assume that $\varphi_1$ preserves the congruence subgroup for $I$.
    The hypotheses of Lemma~\ref{lemma:A4-surjective-coefficient-recovery} are satisfied.
    Hence there exists $\xi\in R_\theta^-$ such that
    \[
        \rho(x_1(\xi))=x_1(\varepsilon).
    \]
    Choose a representative $Y$ of $\psi(x_1(\xi))$ and normalize $Y_{44}=1$.
    Then
    \[
        Y=I_5+\delta(E_{12}-E_{45})+B_1,
        \qquad B_1\in M_5(I).
    \]

    \smallskip
    \noindent
    \emph{The additional root comparison.}
    Because $\xi$ is anti-invariant, its image satisfies
    \[
        [Y,H]=[Y,X_1]=[Y,X_3]=1,\qquad
        w_4Yw_4^{-1}=Y^{-1},
    \]
    and
    \[
        [Y,X_4]=[X_1,V_2Y^{-2}V_2^{-1}]
    \]
    projectively.
    Define the five scalar-free matrix residuals
    \[
    \begin{aligned}
        P_1&=\mathcal P(YH,HY),\\
        P_2&=\mathcal P(w_4Yw_4^{-1},Y^{-1}),\\
        P_3&=\mathcal P([Y,X_1],I_5),\\
        P_4&=\mathcal P([Y,X_3],I_5),\\
        P_5&=\mathcal P([Y,X_4],[X_1,V_2Y^{-2}V_2^{-1}]).
    \end{aligned}
    \]
    All their entries vanish.
    Expanding them in $\mathcal M$ gives the following explicit identity:
    \begin{equation}\label{eq:A4-final-unitary-defect}
    \begin{aligned}
        5e={}&-(P_1)_{31}+\tfrac12(P_1)_{32}+\tfrac12(P_1)_{53}\\
             &-\tfrac12(P_2)_{14}-(P_2)_{24}
               +\tfrac12(P_2)_{25}+(P_2)_{33}\\
             &-\tfrac12(P_3)_{14}+\tfrac12(P_3)_{15}
               +(P_4)_{12}+2(P_4)_{13}\\
             &+\tfrac12(P_5)_{12}+(P_5)_{15}=0.
    \end{aligned}
    \end{equation}
    This is a calculation over $\mathbb Z[1/2]$.
    All coefficients of the fifty old defects and the twenty-five entries of $B_1$ cancel; the constant and linear terms in $\delta$ vanish, and the quadratic coefficient is $5$.
    The ancillary file \texttt{verify\_A4\_normalization\_repair.py} verifies this identity and independently reproduces the preceding $50\times50$ certificate.
    Combining~\eqref{eq:A4-intrinsic-3e} and~\eqref{eq:A4-final-unitary-defect}, we obtain
    \[
        e=2(3e)-5e=0.
    \]
    Equation~\eqref{eq:A4-final-rank-defect} gives $q=0$, and~\eqref{eq:A4-defect-solution-span} now gives $\mathcal M=0$.
    Nakayama's lemma yields $I=0$.
    Thus $X_i=X_i^\delta$, $\delta^2=3\delta=0$, and~\eqref{eq:A4-bar-delta-exact} gives $\bar\delta=-\delta$.

    Finally, the displayed normal forms are unitary for $Q_5$ and also preserve $M$.
    The same matrix-algebra argument used modulo $I$ shows that $M=I_5$.
    Hence $C^\dagger C$ is scalar and $[C]\in G^{(4)}(R)$.
    Substitution in~\eqref{eq:A4-elementary-V2-word} gives $V_2=v_2$, and the stated formulas for $V_3,X_3,X_4$ follow.
\end{proof}

Lemma~\ref{lemma:A4-elementary-delta-normal-form} completes the normalization over $R$.
It gives
\[
    \delta\in J\cap R_\theta^-,
    \qquad \delta^2=3\delta=0,
\]
and shows that the total conjugating matrix $C$ satisfies
\[
    C\in\operatorname{GL}_5(R,J),
    \qquad [C]\in G^{(4)}(R).
\]
For the subsequent argument, write $C_0\coloneqq C$.
By normality of $E^{(4)}(R)$ in $G^{(4)}(R)$~\cite[Corollary~4.3]{SG&DM1}, the normalized map is therefore an automorphism:
\[
    \psi=i_{[C_0]}^{-1}\circ\varphi_1
    \in\operatorname{Aut}(E^{(4)}(R)).
\]
Since $C_0\equiv I_5\pmod J$, it still induces $\mu_k$ on $E^{(4)}(k)$.
Its distinguished root images satisfy
\[
    \psi(u_i)=[X_i^\delta]\qquad(1\leq i\leq4).
\]
The Weyl normalization is preserved, and the short Weyl representatives are now fixed as well:
\[
    \psi(w_1)=w_1,
    \qquad \psi(w_4)=w_4,
    \qquad \psi(H)=H,
    \qquad \psi(v_2)=v_2,
    \qquad \psi(v_3)=v_3.
\]

Thus the hypotheses of
Lemma~\ref{lemma:A4-surjective-coefficient-recovery}
hold with $B=R$ and $\rho=\psi$.
In Subsection~\ref{subsec:parameter-maps-first-annihilators}, we apply this lemma and use the injectivity of $\psi$ to show that the resulting coefficient map is a ring automorphism.


\subsection{Root images and coefficient recovery}
\label{subsec:parameter-maps-first-annihilators}

For each $n\in\{3,4\}$, the preceding normalization gives an automorphism
\[
    \psi\in\operatorname{Aut}(E^{(n)}(R))
\]
with distinguished root images $[X_i^\delta]$, where
\[
    \delta\in J\cap R_\theta^-,
    \qquad \delta^2=3\delta=0.
\]
The hypotheses of Lemma~\ref{lemma:A3-surjective-coefficient-recovery} for $n=3$, and of Lemma~\ref{lemma:A4-surjective-coefficient-recovery} for $n=4$, therefore hold with $B=R$, $\rho=\psi$, and $\varepsilon=\delta$.
We now apply these lemmas and use the injectivity of $\psi$ to recover a ring automorphism.

\begin{prop}[Coefficient recovery and annihilator properties]
    \label{prop:elementary-coefficient-recovery-admissibility}
    The normalized parameter satisfies $\delta\in\mathcal E(R,\theta)$.
    There is a unique ring automorphism $\mu\in\operatorname{Aut}(R)$ commuting with $\theta$ and inducing $\mu_k$ on $R/J$ such that
    \[
        \psi([x_1(a)])
        =\Theta_\delta^{(n)}([x_1(\mu(a))])
        \qquad(a\in R).
    \]
    In particular,
    \[
        \delta R_\theta^-=0,
        \qquad
        \delta I_\theta(R)=0.
    \]
\end{prop}

\begin{proof}
    Apply the coefficient-recovery lemma for the corresponding type with $B=R$ and $\rho=\psi$.
    It gives $\delta\in\mathcal E(R,\theta)$ and a surjective unital ring endomorphism $\mu\colon R\to R$ commuting with $\theta$ and inducing $\mu_k$.
    It also determines the images of all relative root elements by applying $\mu$ to their parameters and then applying $\Theta_\delta^{(n)}$.

    To prove that $\mu$ is injective, suppose that $\mu(a)=0$.
    Then
    \[
        \psi([x_1(a)])=1.
    \]
    Since $\psi$ is injective, $[x_1(a)]=1$.
    The matrix $x_1(a)$ has all diagonal entries equal to $1$ and $(1,2)$-entry equal to $a$, so it is scalar only if $a=0$.
    Thus $\mu$ is injective and hence a ring automorphism.

    For the fixed parameter $\delta$, applying $\Theta_{-\delta}^{(n)}$ to the asserted root-image identity determines $[x_1(\mu(a))]$ for every $a\in R$.
    Since the parametrization $b\mapsto[x_1(b)]$ is injective, this determines $\mu$ uniquely.
    Finally, the two annihilator identities follow from
    Lemma~\ref{lemma:exnl-prmr-space}.
\end{proof}

The symmetric and anti-invariant parameter maps are therefore the restrictions
\[
    \nu=\mu|_{R_\theta},
    \qquad
    \eta=\mu|_{R_\theta^-}.
\]
They are additive bijections of the corresponding parts, and
\begin{equation}\label{eq:mu-decomposition}
    \mu(a)=\nu(a^+)+\eta(a^-)
    \qquad(a\in R).
\end{equation}
For anti-invariant parameters, the coefficient-recovery lemmas give
\begin{equation}\label{eq:elementary-anti-invariant-annihilator}
    \psi([x_1(\xi)])=[x_1(\eta(\xi))],
    \qquad
    \delta\eta(\xi)=0
    \qquad(\xi\in R_\theta^-).
\end{equation}

\begin{prop}[Identification of the residual part]
    \label{prop:identification-residual-part}
    For $n\in\{3,4\}$, the normalized automorphism satisfies
    \[
        \psi=\Theta_\delta^{(n)}\circ\mu
        \qquad\text{on }E^{(n)}(R),
    \]
    where $\mu$ also denotes the induced ring automorphism of the group.
\end{prop}

\begin{proof}
    Since $\mu$ commutes with $\theta$, it induces an automorphism of $E^{(n)}(R)$ by applying $\mu$ to the root parameters.
    In particular, for the short roots of type ${}^2A_4$, it sends
    \[
        (t,u)\longmapsto(\mu(t),\mu(u))
    \]
    in $\mathcal A(R)$.
    The coefficient-recovery lemma for the corresponding type gives
    \[
        \psi(y)=\Theta_\delta^{(n)}(\mu(y))
    \]
    for every relative root element $y$.
    Since these elements generate $E^{(n)}(R)$, the equality holds on the whole group.
\end{proof}


\subsection{Concluding the proof of Theorem~\ref{MT:aut_ele_adj_local}}
\label{subsec:concluding-proof-of-main-theorem-adjoint}

Let $n\in\{3,4\}$ and $\varphi\in\operatorname{Aut}(E^{(n)}(R))$.
The reduction step in Subsection~\ref{subsec:reduction-mod-J}, followed by the normalization in
Subsection~\ref{subsec:normalization_of_u_i_A_3} or
Subsection~\ref{subsec:normalization_of_u_i_A_4}, gives
\[
    \psi=i_{[C_0]}^{-1}\circ i_{g_1}^{-1}\circ\varphi,
    \qquad
    g_1,[C_0]\in G^{(n)}(R).
\]
By Propositions~\ref{prop:elementary-coefficient-recovery-admissibility}
and~\ref{prop:identification-residual-part},
\[
    \psi=\Theta_\delta^{(n)}\circ\mu,
\]
where $\delta\in\mathcal E(R,\theta)$ and
$\mu\in\operatorname{Aut}(R)$ commutes with $\theta$.

Set
\[
    g\coloneqq g_1[C_0]\in G^{(n)}(R),
    \qquad
    \delta'\coloneqq\mu^{-1}(\delta).
\]
Proposition~\ref{prop:action_of_stnrd_aut_on_exp} gives
$\delta'\in\mathcal E(R,\theta)$ and
\[
    \Theta_\delta^{(n)}\circ\mu
    =\mu\circ\Theta_{\delta'}^{(n)}.
\]
Consequently,
\[
    \varphi
    =i_g\circ\Theta_\delta^{(n)}\circ\mu
    =i_g\circ\mu\circ\Theta_{\delta'}^{(n)}.
\]
Since $i_g$ is strictly inner, this is the required factorization and completes the proof of Theorem~\ref{MT:aut_ele_adj_local}.


\section{Automorphisms for All Isogeny Types}
\label{sec:arbitrary-isogeny-types}

Section~\ref{sec:aut_ele_adj_local} classified automorphisms of the
adjoint elementary groups.
We now extend this classification to arbitrary isogeny types,
first for the elementary subgroups and then for the full groups. 

Throughout this section, the standing assumptions on $(R,\theta)$ remain in force.
Let $T_\pi^{(n)}$ denote the standard torus of $G_\pi^{(n)}$.
We use the local-ring decomposition
\begin{equation}\label{eq:full-group-elementary-torus}
    G_\pi^{(n)}=E_\pi^{(n)}T_\pi^{(n)},
\end{equation}
proved in~\cite[Proposition~5.8]{SG&DM1}.


\subsection{Elementary groups}
\label{subsec:aut_of_E_pi}

The natural map
\[
    \lambda_\pi\colon E_\pi^{(n)}\longrightarrow E_{\mathrm{ad}}^{(n)}
\]
is surjective with kernel $Z(E_\pi^{(n)})$~\cite[Lemma~2.7]{EB&DM1}.

\begin{thm}[Theorem~A]\label{thm:exceptional-elementary-arbitrary-isogeny}
    Let $n\in\{3,4\}$ and let $\pi$ be any of the isogeny types considered above.
    Every automorphism of $E_\pi^{(n)}$ has the form
    \[
        \varphi=i_g\circ d\circ\mu\circ\Theta_\delta^{(n)},
    \]
    where $i_g$ is strictly inner, $d$ is diagonal, $\mu\in\operatorname{Aut}(R)$ commutes with $\theta$, and $\delta\in\mathcal E(R,\theta)$.
    Conversely, every such composition is an automorphism.
\end{thm}

\begin{proof}
    Since the center is characteristic, $\varphi$ induces an automorphism of $E_{\mathrm{ad}}^{(n)}$.
    By Theorem~\ref{MT:aut_ele_adj_local}, it has the form
    \[
        \overline\varphi=i_{\overline g}\circ\mu\circ\Theta_\delta^{(n)},
        \qquad \overline g\in G_{\mathrm{ad}}^{(n)}.
    \]
    Write $\overline g=\overline e\,\overline t$ using~\eqref{eq:full-group-elementary-torus}, and choose $g\in E_\pi^{(n)}$ lifting $\overline e$.
    The torus-lifting construction in~\cite[Section~9, proof of Theorem~3.3]{EB&DM1} gives a diagonal automorphism $d$ of $E_\pi^{(n)}$ inducing $i_{\overline t}$.
    Its conjugating torus element is defined over an extension of $(R,\theta)$ and commutes with $T_\pi^{(n)}$.
    Thus~\eqref{eq:full-group-elementary-torus} also shows that $d$ extends to $G_\pi^{(n)}$.

    The ring factor $\mu$ acts naturally on $E_\pi^{(n)}$, and $\Theta_\delta^{(n)}$ lifts compatibly by Proposition~\ref{prop:exp_auto}.
    Hence
    \[
        \alpha\coloneqq i_g\circ d\circ\mu\circ\Theta_\delta^{(n)}
    \]
    induces the same adjoint automorphism as $\varphi$.
    Consequently, $\beta=\alpha^{-1}\circ\varphi$ has the form $\beta(x)=xz(x)$, where
    \[
        z\colon E_\pi^{(n)}\longrightarrow Z(E_\pi^{(n)})
    \]
    is a homomorphism.
    Since $E_\pi^{(n)}$ is perfect~\cite[Corollary~6.6]{SG&DM1}, this homomorphism is trivial, so $\varphi=\alpha$.
    The converse follows from Proposition~\ref{prop:exp_auto} and the definitions of the standard factors.
\end{proof}


\subsection{Full groups}
\label{subsec:aut_of_G_pi}

The distinction between the elementary and full groups is detected by a diagonal automorphism that changes the sign of the exceptional parameter.
Put
\[
    T_0^{(3)}=\operatorname{diag}(1,1,-1,-1),
    \qquad
    T_0^{(4)}=\operatorname{diag}(1,-1,1,-1,1).
\]
Both matrices have determinant one, and
\[
    T_0^{(3)}Q_4(T_0^{(3)})^t=-Q_4,
    \qquad
    T_0^{(4)}Q_5(T_0^{(4)})^t=Q_5.
\]
In type ${}^2A_3$, one has $\sigma(T_0^{(3)})=-T_0^{(3)}$.
Since $-I_4$ acts trivially in the intermediate and adjoint representations, the images $t_\pi^{(3)}$ of $T_0^{(3)}$ belong to $G_\pi^{(3)}$ for $\pi\in\{\mathrm{mid},\mathrm{ad}\}$.
In type ${}^2A_4$, the matrix $T_0^{(4)}$ itself belongs to $G_{\mathrm{sc}}^{(4)}$; write $t_\pi^{(4)}$ for its images in the two isogeny types.

Although $T_0^{(3)}\notin G_{\mathrm{sc}}^{(3)}$, it still induces a diagonal automorphism of $E_{\mathrm{sc}}^{(3)}$.
Indeed, over $S=R[z]/(z^2+1)$ with $\bar z=z$, the diagonal matrix $zT_0^{(3)}$ lies in $\operatorname{SU}_4(S,Q_4)$ and induces the same conjugation.
Let $d_{0,\pi}^{(n)}$ denote the resulting diagonal automorphism of $E_\pi^{(n)}$ in every isogeny type.
Except for $(n,\pi)=(3,\mathrm{sc})$, it is $i_{t_\pi^{(n)}}$ with $t_\pi^{(n)}\in G_\pi^{(n)}$.

\begin{lemma}[Change of sign of the exceptional parameter]
    \label{lem:exceptional-sign-change}
    For every $n\in\{3,4\}$, every isogeny type $\pi$, and every $\delta\in\mathcal E(R,\theta)$,
    \[
        d_{0,\pi}^{(n)}\circ\Theta_\delta^{(n)}\circ
        \bigl(d_{0,\pi}^{(n)}\bigr)^{-1}
        =\Theta_{-\delta}^{(n)}.
    \]
\end{lemma}

\begin{proof}
    The root characters of the two diagonal matrices satisfy
    \[
        \chi^{(3)}(\alpha_2)=-1,
        \qquad \chi^{(4)}(\alpha_1)=-1.
    \]
    Proposition~\ref{prop:action_of_stnrd_aut_on_exp} gives the assertion.
\end{proof}

\begin{prop}\label{prop:exceptional-does-not-extend}
    Suppose that
    \[
        (n,\pi)\in
        \{(3,\mathrm{mid}),(3,\mathrm{ad}),(4,\mathrm{sc}),(4,\mathrm{ad})\}.
    \]
    Then:
    \begin{enumerate}
        \item If $\mathcal E(R,\theta)\ne\{0\}$, then $t_\pi^{(n)}\notin E_\pi^{(n)}$, so $E_\pi^{(n)}\ne G_\pi^{(n)}$.
        \item An exceptional automorphism $\Theta_\delta^{(n)}$ extends to an automorphism of $G_\pi^{(n)}$ only if $\delta=0$.
    \end{enumerate}
\end{prop}

\begin{proof}
    Write $E=E_\pi^{(n)}$, $G=G_\pi^{(n)}$, and $t=t_\pi^{(n)}$.
    Suppose first that $\varphi\in\operatorname{Aut}(G)$ restricts to $\Theta_\delta^{(n)}$ on $E$, and put $s=\varphi(t)$.
    Lemma~\ref{lem:exceptional-sign-change} and the addition law for exceptional parameters give, on $E$,
    \begin{align*}
        i_{st^{-1}}
        &=\Theta_\delta^{(n)}\circ i_t\circ
          \Theta_{-\delta}^{(n)}\circ i_t^{-1}\\
        &=\Theta_\delta^{(n)}\circ\Theta_\delta^{(n)}
         =\Theta_{2\delta}^{(n)}.
    \end{align*}
    The left-hand side is standard, so Proposition~\ref{prop:exceptional-nonstandard} gives $2\delta=0$ and hence $\delta=0$.

    If $t\in E$ and $0\ne\delta\in\mathcal E(R,\theta)$, the same identity holds with $s=\Theta_\delta^{(n)}(t)\in E$ and gives the same contradiction.
    This proves both assertions.
\end{proof}

\begin{thm}[Theorem~B]\label{thm:exceptional-full-arbitrary-isogeny}
    Under the standing assumptions, the automorphisms of the full groups are as follows.
    \begin{enumerate}
        \item For $(n,\pi)=(3,\mathrm{sc})$, every automorphism has the form
        \[
            \varphi=i_g\circ d\circ\mu\circ\Theta_\delta^{(3)},
            \qquad \delta\in\mathcal E(R,\theta).
        \]
        \item For
        \[
            (n,\pi)\in
            \{(3,\mathrm{mid}),(3,\mathrm{ad}),(4,\mathrm{sc}),(4,\mathrm{ad})\},
        \]
        every automorphism is standard and has the form
        \[
            \varphi=i_g\circ d\circ\mu\circ\tau.
        \]
    \end{enumerate}
    Here $i_g$ is strictly inner, $d$ is diagonal, $\mu\in\operatorname{Aut}(R)$ commutes with $\theta$, and $\tau$ is central.
    Conversely, every composition of the indicated form is an automorphism of the corresponding full group.
\end{thm}

\begin{proof}
    We first justify $G_{\mathrm{sc}}^{(3)}=E_{\mathrm{sc}}^{(3)}$.
    An element of the standard torus of $\operatorname{SU}_4(R,Q_4)$ has the form
    \[
        t=\operatorname{diag}(a,b,\bar b^{-1},\bar a^{-1}),
        \qquad a,b\in R^\times,\quad ab=\bar a\bar b.
    \]
    Thus $ab\in R_\theta^\times$, and the formulas of Section~\ref{sec:matrix-realizations} give
    \[
        t=h_1(a)h_2(ab)\in E_{\mathrm{sc}}^{(3)}.
    \]
    Equation~\eqref{eq:full-group-elementary-torus} now gives the equality, and part~(1) follows from Theorem~\ref{thm:exceptional-elementary-arbitrary-isogeny}.

    For part~(2), write $G=G_\pi^{(n)}$ and $E=E_\pi^{(n)}$.
    The subgroup $E$ is characteristic in $G$~\cite[Corollary~A.2]{SG&DM1}.
    Given $\varphi\in\operatorname{Aut}(G)$, Theorem~\ref{thm:exceptional-elementary-arbitrary-isogeny} therefore gives
    \[
        \varphi|_E=\alpha|_E\circ\Theta_\delta^{(n)},
        \qquad \alpha=i_g\circ d\circ\mu\in\operatorname{Aut}(G),
    \]
    where the standard factors extend to $G$ by the construction in its proof.
    Put $\tau=\alpha^{-1}\circ\varphi$.
    Its restriction to $E$ is $\Theta_\delta^{(n)}$, so Proposition~\ref{prop:exceptional-does-not-extend} gives $\delta=0$.
    Hence $\tau$ fixes $E$ pointwise.

    For $a\in G$ and $x\in E$, normality of $E$ gives
    \[
        \tau(a)x\tau(a)^{-1}=\tau(axa^{-1})=axa^{-1}.
    \]
    Therefore
    \[
        a^{-1}\tau(a)\in C_G(E)=Z(G),
    \]
    where the last equality is~\cite[Theorem~4.4]{SG&DM1}.
    Thus $\tau(a)=az(a)$ for a homomorphism $z\colon G\to Z(G)$, so $\tau$ is central and $\varphi=\alpha\circ\tau$.
    The converse follows from the definitions and Proposition~\ref{prop:exp_auto}.
\end{proof}

Together with~\cite{EB&DM1}, this completes the automorphism classification for type ${}^2A_\ell$ with $\ell\geq3$ under the standing local-ring assumptions.
The low-rank exception is accounted for by the explicit residual cocycles of Subsection~\ref{subsec:expl_aut}.
Their change of sign under diagonal conjugation explains the obstruction to extending a nontrivial exceptional factor from the elementary subgroup to the full group.

\section*{Use of AI-assisted tools}

The authors used ChatGPT (OpenAI) throughout the research and preparation of this paper to assist with discussing mathematical ideas, checking arguments, improving the exposition, and writing and debugging the ancillary verification code.
The authors retain full responsibility for the mathematical results and proofs, the computational methods and their implementation, and the final manuscript.



\end{document}